\documentclass[10pt]{amsart}
\usepackage[english]{babel}
\usepackage{amssymb}
\usepackage{enumitem}
\usepackage{hyperref}
\usepackage[initials]{amsrefs}
\usepackage[all]{xy}
\SelectTips{cm}{}

\newcommand{\bbN}{{\mathbb N}}
\newcommand{\bbQ}{{\mathbb Q}}

\newcommand{\bbR}{{\mathbb R}}

\newcommand{\bbC}{{\mathbb C}}

\newcommand{\supp}{\operatorname{supp}}

\newcommand{\Ker}{\operatorname{Ker}}

\newcommand{\Stab}{\operatorname{Stab}}

\newcommand{\Fix}{\operatorname{Fix}}

\newcommand{\Aut}{\operatorname{Aut}}

\newcommand{\GL}{\operatorname{GL}}

\newcommand{\ec}{/\!\!/}

\newtheorem{theorem}{Theorem}[section]
\newtheorem{lemma}[theorem]{Lemma}

\newtheorem{cor}[theorem]{Corollary}
\newtheorem{proposition}[theorem]{Proposition}
\newtheorem{prop}[theorem]{Proposition}

\theoremstyle{definition}

\newtheorem{defn}[theorem]{Definition}

\newtheorem{example}[theorem]{Example}

\newtheorem{remark}[theorem]{Remark}

\numberwithin{equation}{section}
\begin{document}

\title{Algebraic Representations of Ergodic Actions and Super-Rigidity}

\author{Uri Bader}
%\address{Technion, Haifa}
%\email{uri.bader@gmail.com}

\author{Alex Furman}
%\address{University of Illinois at Chicago, Chicago}
%\email{furman@math.uic.edu}

%\thanks{U.B. and A.F. were supported in part by the BSF grant 2008267.}
%\thanks{U.B was supported in part by the ISF grant 704/08.}
%\thanks{A.F. was supported in part by the NSF grants DMS 0604611 and 0905977.}

%\subjclass[2000]{Primary 20F67; Secondary 55N99}
%\keywords{Hyperbolic groups, measure equivalence, simplicial volume}

\maketitle

\begin{abstract}
We revisit Margulis-Zimmer Super-Rigidity and provide some generalizations.
In particular we obtain super-rigidity results for lattices in higher-rank groups or product of groups, targeting at algebraic groups over arbitrary fields with absolute values.
We also obtain cocycle super-rigidity results for a wide class of groups with respect to mixing actions.
Our approach is based on a systematic study of algebraic representations of ergodic actions.
\end{abstract}

\section{Introduction}

In this paper we study systematically the phenomenon of super-rigidity discovered by Margulis in the late 1970's
and later extended by Zimmer.
The two monographs \cite{margulis-book} and \cite{zimmer-book},
and in particular the celebrated Margulis Super-Rigidity Theorem \cite[Theorem~VII.5.6]{margulis-book}
and Zimmer Cocycle Super-Rigidity Theorem \cite[Theorem 5.2.5]{zimmer-book},
had (and still have) a tremendous impact on various mathematical subjects
and on a large community of researchers.
Specifically, the authors of these notes are greatly influenced by Margulis and Zimmer and their mathematical methods and perspectives.

Our method enables us to prove the following extension of the above mentioned Margulis Super-Rigidity Theorem.
An important ingredient in our proof is a result developed together with Jean L\'{e}cureux and Bruno Duchesne
which will appear soon in \cite{BDL}.

\begin{theorem}[Margulis super-rigidity for arbitrary fields] \label{marguliscor}
Let $l$ be a local field.
Let $T$ to be the $l$-points of a connected almost-simple algebraic group defined over $l$.
Assume that the $l$-rank of $T$ is at least two.
Let $\Gamma<T$ be a lattice.

Let $k$ be a field with an absolute value.
Assume that as a metric space $k$ is complete.
Let $G$ be the $k$-points of an adjoint form simple algebraic group defined over $k$.
Let $\delta:\Gamma \to G$ be a homomorphism.
Assume $\delta(\Gamma)$ is Zariski dense in $G$ and unbounded.
Then there exists a continuous homomorphism $d:T\to G$
such that $\delta=d|_{\Gamma}$.
\end{theorem}

The conclusion of Theorem~\ref{marguliscor} holds also for irreducible lattices in semi-simple groups and $S$-arithmetic groups.
More generally, it holds for any irreducible lattice in a product of groups.
Recall that a lattice in a product of groups is called irreducible if it has a dense image in each factor.

\begin{theorem}[Super-Rigidity for lattices in products] \label{thm:lattice}
Let $T=T_1\times T_2\times \cdots\times T_n$ be a product of lcsc groups,
and let $\Gamma<T$ be an irreducible lattice.

Let $k$ be a field with an absolute value.
Assume that as a metric space $k$ is complete.
Let $G$ be the $k$-points of an adjoint form simple algebraic group defined over $k$.
Let $\delta:\Gamma \to G$ be a homomorphism.
Assume $\delta(\Gamma)$ is Zariski dense in $G$ and unbounded.
Then there exists a continuous homomorphism $d:T\to G$
such that $\delta=d|_{\Gamma}$.
Such a homomorphism must factor through the projection $T\to T_i$ for some $i \leq n$.
\end{theorem}

Theorem~\ref{thm:lattice} extends previous theorems by Monod \cite{Monod-products} and Gelander,
Karlsson and Margulis \cite{GKM} in which similar results are proven for uniform lattices (or under certain integrability conditions).
We remark that both \cite{Monod-products} and \cite{GKM} consider more general settings than ours.

For both Theorem~\ref{marguliscor} and Theorem~\ref{thm:lattice} we prove
analogous in the realm of cocycle super-rigidity a la Zimmer, see Theorem\ref{zimmercor} and Theorem~\ref{prod-cocycle}.
In this setting one considers an ergodic action of the group $T$ on a Lebesgue space $X$
and a cocycle $c:T\times X\to G$,
that is a measurable map satisfying the a.e.\ identity $c(tt',x)=c(t,t'x)c(t',x)$.
The notion "cocycle super-rigidity theorem" refers to a theorem stating that under certain conditions
any such a cocycle is cohomologous to a homomorphism,
that is there exists a homomorphism $d:T\to G$ and a measurable map $\phi:X\to G$
satisfying the a.e.\ identity $c(t,x)\phi(x)=\phi(tx)d(t)$.
%Our method of proving cocycle super-rigidity relies on two main assumptions:
%we assume that the acting group $T$ is generated by successively commuting subgroups (among which at least one is amenable)
%and that the action of $T$ on $X$ is mixing.
%The following theorem is a generalization of Zimmer Cocycle Super-Rigidity Thoerem.

\begin{theorem}[Generalized cocycle super-rigidity] \label{prezimmer}
Let $T$ be a locally compact second countable group.
Assume $T$ is generated as a topological group by the closed, non-compact subgroups $T_0,T_1,T_2,\ldots$
(for any finite number or countably many $T_i$'s).
Assume $T_0$ is amenable and for each $i=1,2,\ldots$ the groups $T_{i-1}$ and $T_i$ commute.
Let $X$ be a $T$-Lebesgue space
with a finite invariant mixing measure.

Let $k$ be a local field.
Let $G$ be the $k$-points of an adjoint form simple algebraic group defined over $k$.
Let $c:T\times X \to G$ be a measurable cocycle.
Assume that $c$ is not cohomologous to a cocycle taking values in a proper algebraic subgroup or a bounded subgroup of $G$.
Then there exists a continuous homomorphism $d:T\to G$
such that $c$ is cohomologous to $d$.
\end{theorem}

We remark that a when the groups $T_i$ appearing in Theorem~\ref{prezimmer} pairwise commute,
the mixing assumption on the action of $T$ on $X$ could be relaxed to the assumption of irreducibility with the same conclusion,
see Theorem~\ref{prod-cocycle}.
In case $T$ is a simple algebraic group over a local field,
mixing is automaticly implied by ergodicity (by Howe-Moore theorem) and the property of generation by successively commuting subgroups
coincides with the standard "higher-rank" property (in the proof of Theorem~\ref{zimmercor} we make this remark precise).
We note that there is a large variety of other groups (typically countable) that satisfy the
generation by successively commuting subgroups property for which Theorem~\ref{prezimmer} applies
once one forces a mixing assumption.
Prominent examples are
mapping class groups (of genus 2 or higher),
Automorphism groups of free groups (of rank 3 or higher)
and right angeled Artin groups (defined by a connected graph).

\subsection{Acknowledgment}
It is our pleasure to thank Bruno Duchesne and
Jean L\'{e}cureux
for their contribution.
We are grateful to Michael Puschnigg for spotting an inaccuracy in the definition of a morphism of bi-algebraic representations in an earlier draft.
We would also like to thank Tsachik Gelander for numerous discussions
as well as for his exposition of a preliminary version of this work in his Zurich ETH Nachdiplom lectures.

\section{Algebraic varieties as Polish spaces} \label{alg perlim}

In this section
we fix a field $k$ with an absolute value $|\cdot|$
as defined and discussed in \cite{valued}.
We assume that the absolute value is non-trivial and that $(k,|\cdot|)$ is complete and separable
(in the sense of having a countable dense subset).
The theory of manifolds over such fields is developed in \cite[Part II, Chapter I]{serre}.
We also fix a $k$-algebraic group ${\bf G}$.
we will discuss the category of $k$-${\bf G}$-varieties.
A $k$-${\bf G}$-variety is a $k$-variety endowed with an algebraic action of ${\bf G}$ which is defined over $k$.
A morphism of such varieties is a $k$-morphism which commutes with the ${\bf G}$-action.
% We denote this category by $\mathcal{A}(k,{\bf G})$.
By a $k$-coset variety we mean a variety 
of the form
%which is $k$-isomorphic to the variety 
${\bf G}/{\bf H}$ for some $k$-algebraic subgroup ${\bf H} < {\bf G}$
(see \cite[Theorem 6.8]{borel}).
% We denote the full subcategory of $\mathcal{A}(k,{\bf G})$ consisting of $k$-coset varieties by $\mathcal{A}_c(k,{\bf G})$.

Each $k$-${\bf G}$-variety gives rise to a topological space: $V={\bf V}(k)$ endowed with its $(k,|\cdot|)$-topology.
Topological notions, unless otherwise said, will always refer to this topology.
In particular $G={\bf G}(k)$ is a topological group.

Recall that a topological space is called Polish if it is separable and completely metrizable.
For a good survey on the subject we recommend \cite{kechris}.
We mention that the class of Polish spaces is closed under countable disjoint unions and countable products.
A $G_\delta$ subset of a Polish spaces is Polish so, in particular, a locally closed subset of a Polish space is Polish.
A Hausdorff space which admits a finite open covering by Polish open sets is itself Polish.
Indeed,
such a space is clearly metrizable (e.g.\ by Smirnov metrization theorem)
so it is Polish by
Sierpinski theorem \cite[Theorem 8.19]{kechris} which states that the image of an open map from a Polish space to a separable metrizable space is Polish.
Sierpinski theorem also implies that for a Polish group $K$ and a closed subgroup $L$, the quotient topology on $K/L$ is Polish.
Effros theorem \cite[lemma 2.5]{effros} says that the quotient topology on $K/L$ is the unique $K$-invariant Polish topology on this space.

\begin{proposition} \label{polishing}
The $k$-points of a $k$-variety form a Polish space.
In particular $G$ is a Polish group.
If ${\bf V}$ is a $k$-${\bf G}$-variety
then the $G$ orbits in $V$ are locally closed.
For $v\in V$ the orbit ${\bf G}v$ is a $k$-subvariety of ${\bf V}$.
The stabilizer ${\bf H}<{\bf G}$
is defined over $k$ and the orbit map ${\bf G}/{\bf H}\to {\bf G}v$ is defined over $k$.
Denoting $H={\bf H}(k)$, the induced map $G/H \to Gv$ is a homeomorphism,
when $G/H$ is endowed with the quotient space topology and $Gv$ is endowed with the subspace topology.
\end{proposition}

\begin{proof}
Since $k$ is complete and separable it is Polish and so is the affine space $\mathbb{A}^n(k) \simeq k^n$.
The set of $k$-points of a $k$-affine variety is closed in the $k$-points of the affine space, hence it is a Polish subspace.
It follows that the set of $k$-points of any $k$-variety is a Polish space,
as this space is a Hausdorff space which admits a finite open covering by Polish open sets - the $k$-points of its $k$-affine charts.

The fact that the $G$ orbits in $V$ are locally closed is
proven in the appendix of \cite{b-z}.
Note that in \cite{b-z} the statement is claimed only for non-archimedean local fields, but the proof is actually correct for any field with
a complete non-trivial absolute value, which is the setting of \cite[Part II, Chapter III]{serre}
on which \cite{b-z} relies.

For $v\in V$ the orbit ${\bf G}v$ is a $k$-subvariety of ${\bf V}$ by \cite[Proposition 6.7]{borel}.
The stabilizer ${\bf H}<{\bf G}$
is defined over $k$ by \cite[Proposition 1.7]{borel}
and we get an orbit map which is defined over $k$ by \cite[Theorem 6.8]{borel}.
Clearly $H$ is the stabilizer of $v$ in $G$ and the orbit map restricts to a continuous map from $G/H$ onto $Gv$.
Since $Gv$ is a Polish subset of $V$, as it is locally closed, we conclude by Effros theorem that the latter map is a homeomorphism.
\end{proof}

%%%%%%%%%@@@

\section{Measurable cocycles} \label{erg perlim}

In this section we set our ergodic theoretical frame work and notations.
All the results we present here could be found in \cite{zimmer-book}, but for the reader's convenient
we gather them here along with self contained proofs.

By a Borel space we mean a set endowed with a $\sigma$-algebra.
A standard Borle space is a Borel space which is isomorphic as such to the underlying Borel space of a Polish topological space.
A Lebesgue space is a standard Borel space endowed with a measure class.
In particular, every coset space of an lcsc group is a Lebesgue space, when endowed with its Haar measure class.
Unless otherwise said we will always regard the Haar measure class when considering lcsc groups or their coset spaces as Lebesgue spaces.
Given an lcsc group $T$, a $T$-Lebesgue space is a Lebesgue space endowed with a measure class preserving action of $T$, defined up to null sets.
%It is well known that every $T$-Lebesgue space could be realized as continuous action of $T$ on a compact metrizable space endowed with an invariant %Radon measure class.
For a Polish group $G$ and  a $T$-Lebesgue space $X$, a measurable map $c:T\times X\to G$ is called a cocycle if for every $t,t' \in T$,
for a.e.\ $x\in X$, $c(tt',x)=c(t,t'x)c(t',x)$.
Two cocycles $c,c'$ are called cohomologous if there exists $\phi\in L^0(X,G)$ such that for every $t\in T$, for a.e.\ $x\in X$,
$c(t,x)\phi(x)=\phi(tx)c'(t,x)$.

\begin{example} \label{stand}
If $X=T/\Gamma$ for some closed subgroup $\Gamma<T$, we can choose a Borel section $\sigma:X \to T$
to the obvious map $\alpha:T\to X=T/\Gamma$.
The map $m:X\times \Gamma\to T$ given by $m(x,\gamma)=\sigma(x)\gamma$ is a Borel isomorphism
and we denote by $\pi:T\to \Gamma$ the map $\sigma(x)\gamma \mapsto \gamma$.
The map $\kappa:T\times X \to \Gamma$ given by $\kappa(t,x)= \pi(t\sigma(x))$
is easily checked to be a cocycle, which depends on the choice of the section $\sigma$ only up to cohomology.
It is called the standard cocycle of $\Gamma$ in $T$.
Note that for every  $x\in X$ and $\gamma\in \Gamma$ we have
$\gamma^{\sigma(x)}x=x$ and for every
$t\in T$,
\begin{equation} \label{keq}
\kappa(t\gamma^{\sigma(x)},x)=\pi(t\gamma^{\sigma(x)}\sigma(x))=\pi(t\sigma(x)\gamma)=\pi(t\sigma(x))\gamma=
\kappa(t,x)\gamma. 
\end{equation}
\end{example}

Two cocycles $c,c'$ are called cohomologous if there exists a.e.\ defined measurable map $\phi:X\to G$ such that for every $t\in T$, a.e.\ $x\in X$,
$c(t,x)\phi(x)=\phi(tx)c'(t,x)$.

\begin{prop}[cf.\ {\cite[Proposition 4.2.16]{zimmer-book}}] \label{lattice-vs-cocycle}
Fix an lcsc group $T$ and a closed subgroup $\Gamma<T$.
Let $X=T/\Gamma$, choose a section $\sigma:X\to T$ and let $\kappa$ be the associated standard cocycle.
Fix a Polish group $G$ and continuous homomorphisms $d:T\to G$ and $\delta:\Gamma\to G$.
Assume the cocyle $\delta\kappa=\delta\circ \kappa$ is cohomologous to the homomorphism $d$.
Then there exists $g\in G$ such that $\delta=d^g|_\Gamma$.
\end{prop}

\begin{proof}
We are given a conull set $X'\subset X$ and a measurable map $\phi:X'\to G$ such that for every $t\in T$, a.e.\ $x\in X'$,
$\delta\kappa(t,x)\phi(x)=\phi(tx)d(t)$.
By Fubini Theorem there exists $x\in X'$ and a conull set $T'\subset T$ such that for every  $t\in T'$
we have $tx\in X'$ and $\delta\kappa(t,x)\phi(x)=\phi(tx)d(t)$.
We fix such an $x$.
Fix also $\gamma\in \Gamma$.
For $t$ in the conull set $T'\cap T'(\gamma^{\sigma(x)})^{-1}$ we obtain
both
\[ \delta\kappa(t,x)\phi(x)=\phi(tx)d(t) \quad \mbox{and} \quad \delta\kappa(t\gamma^{\sigma(x)},x)\phi(x)=\phi(tx)d(t)d(\gamma)^{d\sigma(x)}. \]
Using equation~(\ref{keq}) we now get
\[  \phi(tx)d(t)\phi(x)^{-1}\delta(\gamma)\phi(x)
=\delta\kappa(t, x)\phi(x)\phi(x)^{-1}\delta(\gamma)\phi(x)
=
\delta(\kappa(t, x)\gamma)\phi(x)
\]\[
=\delta\kappa(t\gamma^{\sigma(x)},x)\phi(x)=
\phi(t x)d(t)d(\gamma)^{d(\sigma(x))}
 \]
which gives
$\phi(x)^{-1}\delta(\gamma)\phi(x)=d(\gamma)^{d\sigma(x)}$, thus $\delta=d^{\phi(x)d\sigma(x)}|_\Gamma$.
\end{proof}

%%%%%%%%%%%%%%%%&&&&&&&&&&&&&&

Given a Lebesgue space $X$ and a Borel space $V$ we let $L^0(X,V)$ be the space of classes of measurable maps from $X$ to $V$, defined up to null sets.
If $X$ is a $T$-Lebesgue space for an lcsc group $T$, and fixing a Polish group $G$, a cocycle $c:T\times X\to G$ and a $G$-Borel space $V$,
%Fixing a cocycle $c:T\times X\to G$, for every $G$-Borel space $V$
we obtain an action of $T$ on the space $L^0(X,V)$ given by defining
for $t\in T$ and $\alpha\in L^0(X ,V)$,
$t\alpha\in L^0(X ,V)$ by the formula $t\alpha(x)=c(t,t^{-1}x)\alpha(t^{-1}x)$.
The elements of the fixed point set $ L^0(X ,V)^T$ are called $c$-equivariant maps.
These are a.e.\ defined measurable maps $\phi:X\to V$ satisfying for every $t\in T$ and a.e.\ $x\in X$, $\phi(tx)=c(t,x)\phi(x)$.

\begin{prop}[cf.\ {\cite[Example 4.2.18(b)]{zimmer-book}}] \label{proper-cohom}
Fix an lcsc group $T$ and a $T$-Lebesgue space $X$.
Let $G$ be a Polish group and $c:T\times X\to G$ a measurable cocycle.
Let $H<G$ be a closed subgroup.
Then there exists a $c$-equivariant map $\phi:X\to G/H$ iff $c$ is cohomologous to a cocycle taking values in $H$.
\end{prop}

\begin{proof}
For a a $c$-equivariant map $\phi:X\to G/H$ and Borel section $\sigma:G/H\to G$
to the obvious map $\alpha:G\to G/H$, one sets $\psi=\sigma\phi$ and checks easily that
$\psi(tx)^{-1}c(t,x)\psi(x)$ defines a cocycle that takes values in $H$.
Conversely, if $\psi:X\to G$ is such that $\psi(tx)^{-1}c(t,x)\psi(x)$ takes values in $H$
then $\phi=\alpha\psi$ is clearly $c$-equivariant.
\end{proof}

%It is tempting to consider the space $X\times V$ as a $T$-space, with the action given by $t(x,v)=(tx,c(t,x)v)$.
%Strictly speaking, this construction does not make sense when $c$ is only defined for every $t\in T$, a.e.\ $x\in X$.
%It does make perfect sense when $X$ is a $T$-Borel space and $c:T\times X\to G$ is a Borel map defined for every $t\in T$ and every $x\in X$,
%satisfying for every $t,t'\in T$ and every $x\in X$, $c(tt',x)=c(t,t'x)c(t',x)$.
%In that case $X\times V$ becomes a $T$-Borel space, and every map $\phi\in L^0(X,V)^T$
%gives a $T$ equivariant map $X\to X\times V$ by $x\mapsto (x,\phi(x))$.

\begin{prop}[cf.\ {\cite[Lemma 5.2.6]{zimmer-book}}] \label{Gamma-c-ergodic}
Fix an lcsc group $T$ and a closed subgroup $\Gamma<T$.
Let $X=T/\Gamma$, choose a section $\sigma:X\to T$ and let $\kappa$ be the associated standard cocycle.
Fix a Polish group $G$
and a continuous homomorphism $\delta:\Gamma\to G$.
%Fix a space $V$ and a $\delta\kappa$-equivariant map $\phi$.
%Then for a.e.\ $x\in X$, $\phi(x)\in V^{\delta(\Gamma)}$.
%In particular, 
Assume $\delta\kappa$ is cohomologous to a cocyle taking values in a closed subgroup $H<G$.
Then 
for some $g\in G$, $\delta(\Gamma)^g<H$.
\end{prop}

\begin{proof}
Note that the cocycle $\delta\kappa:T\times X\to G$ is an everywhere defined Borel cocycle,
thus $X\times G/H$ becomes a $T$-Borel space,  with the action given by $t(x,v)=(tx,\delta\kappa(t,x)v)$.
By Proposition~\ref{proper-cohom} there exists a $\delta\kappa$-equivariant map $\phi:X\to G/H$.
Observe that $x\mapsto (x,\phi(x))$ defines a $T$-equivariant map $X\to X\times G/H$.
Fix a generic $x\in X$.
For $\gamma\in \Gamma$, $\gamma^{\sigma(x)}x=x$, hence
$\gamma^{\sigma(x)}$ also stabilizes its image $(x,\phi(x))$.
By equation~(\ref{keq}),
\[ (x,\phi(x)) = \gamma^{\sigma(x)}(x,\phi(x))=(\gamma^{\sigma(x)}x,\delta\kappa(\gamma^{\sigma(x)},x)\phi(x))=
(\gamma^{\sigma(x)}x,\delta(\gamma)\phi(x)), \]
hence $\phi(x)=\delta(\gamma)\phi(x)$.
Thus $\delta(\Gamma)<\Stab_G(\phi(x))$.
The latter is a is a conjugate of $H$.
\end{proof}

\begin{prop}[cf.\ {\cite[Lemma 5.2.11]{zimmer-book}}] \label{coset contraction-p}
Let $T$ be an lcsc group and $X$ an ergodic $T$-Lebesgue space.
Let $G$ be a Polish group acting continuously on a Polish space $V$.
Assume all $G$ orbits are locally closed in $V$.
Let $c:T\times X\to G$ be a measurable cocycle.
Let $\phi\in L^0(X,V)^T$ be a $c$-equivariant map.
Then for some $v\in V$
there exists a $c$-equivariant map $\phi'\in L^0(X,G/H)$,
where $H$ is the stabilizer of $v$ in $G$,
such that $\phi=i\circ \phi'$
where $i:G/H \to V$ is the orbit map $gH \mapsto gv$.
\end{prop}

\begin{proof}
Consider the orbit space
$V/G$ endowed with the quotient topology and Borel structure.
The map $X\to V \to V/G$ is Borel.
We denote by $\mu$ the measure on $V/G$ obtained by pushing the measure on $X$.
Since $V$ is Polish it has a countable basis thus so does $V/G$.
Let $\{U_n~|~n\in \bbN\}$ be sequence of subsets of $V/G$ consisting of the elements of a countable basis and their complements.
Set
\[ U=\bigcap \{U_n~|~n\in \bbN,~\mu(U_n^c)=0\}. \]
Then $\mu(U^c)=0$ and in particular $U$ is non-empty.
We claim that $U$ is a singleton.
Indeed,
since every $G$-orbit is locally closed in $V$, $V/G$ is $T_0$,
thus if $U$ would contain two distinct points we could find a basis set which separate them,
but by ergodicity either it or its complement would be $\mu$-null, which contradicts the definition of $U$.

Fixing $v\in V$ which is in the preimage of $U$ we conclude that $\phi(X)$ is essentially contained in $Gv$.
Letting $H<G$ be the stabilizer of $v$,
the orbit map $i:G/H \to Gv$ is a homeomorphism by
Effros theorem \cite[lemma 2.5]{effros}.
We are done by setting $\phi'=(i|_{G/H})^{-1} \circ \phi$.
%\[ X  \xrightarrow{\phi} Gv  \xrightarrow{(i^{-1})_k} G/H = {\bf G}(k)/{\bf H}(k) \subset ({\bf G}/{\bf H})(k). \]
\end{proof}

\section{Amenability and Metric Ergodicity} \label{A+ME}

\begin{defn}
An action of an lcsc group $T$ on a Lebesgue space $X$ is said to be metrically ergodic is for every metric space $V$
endowed with a continuous isometric action of $T$ we have that the obvious inclusion $V^T \hookrightarrow L^0(X,V)^T$,
given by sending a $T$-fixed point to the constant map from $X$ to this point, is onto.
Equivalently, every $T$-equivariant measurable function from $X$ to $V$ is essentially constant.
\end{defn}

\begin{example} \label{example}
For a closed subgroup $T_0<T$, it is easy to see that the action of $T$ on $T/T_0$ is metrically ergodic iff $T/T_0$ admits
no $T$-invariant continuous semi-metric, but the trivial one.
\end{example}

\begin{lemma} \label{me-e}
Let $T$ be an lcsc group and $X,B$ be $T$-Lebesgue spaces.
Assume the action on $X$ is ergodic and probability measure preserving and the action on $B$ is metrically ergodic.
Then the diagonal action on $B\times X$ is ergodic.
\end{lemma}

\begin{proof}
For $f\in L^{\infty}(B\times X)^T$, using Fubini theorem, we define $F:B\to L^\infty(X) \subset L^2(X)$
by $F(b)(x)=f(b,x)$.
$F$ is easily checked to be $T$-equivariant.
The image of $F$ must be $T$-invariant, by the metric ergodicity of $B$, as the $T$-action on $L^2(X)$ is continuous.
By ergodicity of $X$ this image is a constant function, thus $f$ is constant.
\end{proof}

Given a standard Borel spaces $V$, a Lebesgue space $X$ and an essentially surjective Borel map $\pi: V\to X$
(that is, $\pi$ is measurable for a Borel model of $X$ and its image has full measure),
we denote by $L^0(\pi)$ the set of all equivalence classes of sections of $\pi$, defined up to null sets in $X$.

\begin{defn}[{\cite[Definition~4.3.1]{zimmer-book}}]
An action of an lcsc group $T$ on a Lebesgue space $X$ is said to be amenable if for every $T$-Borel space $V$
and an essentially surjective $T$-equivariant Borel map $\pi: V\to X$ with compact convex fibers,
such that the $T$ action restricted to the fibers is by continuous affine maps,
one has $L^0(\pi)^T\neq \emptyset$. That is, every $T$-Borel bundle of convex compact sets over $X$ admits an invariant measurable section.
\end{defn}

\begin{example}[{\cite[Proposition~4.3.2]{zimmer-book}}] \label{amen subgroup}
For a closed subgroup $T_0<T$, the action of $T$ on $T/T_0$ is amenable iff $T_0$ is amenable.
\end{example}

\begin{theorem}[Kaimanovich-Zimmer] \label{boundary}
For every lcsc group $T$ there exists a $T$-Lebesgue space $B$ possessing the following two properties:
the action of $T$ on $B$ is amenable and the diagonal action of $T$ on $B\times B$ is metrically ergodic.
\end{theorem}

\begin{remark}
In \cite{BF-metric} a slightly stronger theorem is proven, but we will not discuss this generalization here.
\end{remark}

\begin{proof}[On the proof of Theorem~\ref{boundary}]
In \cite{kaimanovich} the weaker statement that every group posses a strong boundary in the sense of \cite{BM} is proven,
but the same proof actually proves Theorem~\ref{boundary},
as explained in \cite[Remark~4.3]{GW}.
\end{proof}

%The following result, that shows that both amenability and metric ergodicity behave well under taking products,
%is essential to the proof of Super-Rigidity for products.

\begin{prop} \label{product-stable}
Let $T_1,\ldots,T_n$ be lcsc groups and set $T=T_1\times\cdots\times T_n$.
For each $i$ let $B_i$ be a $T_i$-Lebesgue space and set $B=B_1\times \cdots\times B_n$.
If for each $i$ the action of $T_i$ on $B_i$ is amenable then so is the action of $T$ on $B$.
If for each $i$ the action of $T_i$ on $B_i$ is metrically ergodic then so is the action of $T$ on $B$.
\end{prop}

\begin{proof}
We will prove the statements for $n=2$ which is clearly enough, by an inductive argument.

Let $\pi:C\to B$ be a
$T$-Borel bundle of convex compact sets over $B$.
For every $b_1\in B_1$ we let $C_{b_1}=\pi^{-1}(\{b_1\}\times B_2)$ and
$L^0(\pi|_{C_{b_1}})$ be the corresponding space of sections.
We view these spaces as a convex compact bundle over $B_1$ and denote by $L^0(B_1,L^0(\pi|_{C_{b_1}}))$ its space of section.
Its obvious identification with $L^0(\pi)$ gives an identification
$L^0(\pi)^T\simeq L^0(B_1,L^0(\pi|_{C_{b_1}})^{T_2})^{T_1}$.
The right hand side is non-empty by our amenability assumptions, hence so is the left hand side.

Let $V$ be a $T$-metric space and $\phi:B\to V$  a measurable $T$-map.
For a.e.\ $b_1\in B_1$, $\phi_{\{b_1\}\times B_2}$ is essentially constant, as $B_2$ is $T_2$ metrically ergodic,
thus $\phi$ reduces to a map $\phi':B_1\to V$ which is again essentially constant, as $B_1$ is $T_1$ metrically ergodic.
Thus $\phi$ is essentially constant.
\end{proof}

\section{Algebraic representation of ergodic actions} \label{alg-rep}

Throughout this section we fix
\begin{itemize}
\item an lcsc group $T$,
\item an ergodic $S$-Lebesgue space $Y$,
\item a field $k$ with an absolute value which is separable and complete (as a metric space),
\item a $k$-algebraic group ${\bf G}$,
\item a measurable cocycle
$c:T \times Y \to G$,
where $G={\bf G}(k)$ is regarded as a Polish group (Proposition~\ref{polishing}).
\end{itemize}

\begin{defn}
Given all the data above, an algebraic representation of $Y$
consists of the following data
\begin{itemize}
\item a $k$-${\bf G}$-algebraic variety ${\bf V}$,
\item a measurable map $\phi:Y \to {\bf V}(k)$ such that for every $t\in T$ and almost every $y\in Y$,
\[ \phi(ty)=c(t,y)\phi(y). \]
\end{itemize}
We abbreviate the notation by saying that ${\bf V}$ is an algebraic representation of $Y$,
and denote $\phi$ by $\phi_{\bf V}$ for clarity.
A morphism from the algebraic representation ${\bf U}$ to the algebraic representation ${\bf V}$ consists of
\begin{itemize}
\item a $k$-algebraic map $\psi:{\bf U}\to {\bf V}$ which is ${\bf G}$ equivariant,
and such that $\phi_{\bf V}=\psi\circ \phi_{\bf U}$.
\end{itemize}
An algebraic representation ${\bf V}$ of $Y$ is said to be a coset algebraic representation
if in addition
${\bf V}={\bf G}/{\bf H}$ for some $k$-algebraic subgroup ${\bf H}<{\bf G}$.
\end{defn}

\begin{prop} \label{coset contraction}
%Assume $Y$ is $T$ ergodic and 
Let ${\bf V}$ be an algebraic representation of $Y$.
Then there exists a coset algebraic representation ${\bf G}/{\bf H}$
and a morphism of representations from ${\bf G}/{\bf H}$ to ${\bf V}$,
that is a $k$-${\bf G}$-algebraic map $i:{\bf G}/{\bf H}\to {\bf V}$
such that $\phi_{\bf V}=i\circ \phi_{{\bf G}/{\bf H}}$.
\end{prop}

\begin{proof}
Denoting $V={\bf V}(k)$ and using Proposition~\ref{polishing},
we obtain by Proposition~\ref{coset contraction-p} the existence of some $v\in V$
and a $c$-equivariant map $\phi'\in L^0(Y,G/H)$,
where $H$ is the stabilizer of $v$ in $G$,
such that $\phi=i\circ \phi'$
where $i:G/H \to V$ is the orbit map $gH \mapsto gv$.
Clearly $H={\bf H}(k)$, where ${\bf H}$ is the stabilizer of $v$ in ${\bf G}$,
and $i$ extends to a $k$-${\bf G}$-algebraic map ${\bf G}/{\bf H}\to {\bf V}$.
We are done by extending the codomain of $\phi'$ to ${\bf G}/{\bf H}(k)$ vis the natural imbedding $G/H \hookrightarrow {\bf G}/{\bf H}(k)$.
\end{proof}

%It is clear that the collection of algebraic representations of $X$ and their morphisms form a category.
In \cite[Definition~9.2.2]{zimmer-book}
Zimmer defined the notion "algebraic hull of a cocyle".
We will not discuss this notion here,
but we do point out its close relation with the following theorem (to be precise, it coincides with the group ${\bf H}_0$
appearing in the proof below).

\begin{theorem}[cf.\ {\cite[Proposition~9.2.1]{zimmer-book}}] \label{a-initial}
%Assume $Y$ is $T$ ergodic.
The category of algebraic representations of $Y$ has an initial object.
Moreover, there exists an initial object which is a coset algebraic representation.
\end{theorem}

\begin{proof}
We consider the collection
\[ \{{\bf H}<{\bf G}~|~{\bf H}\mbox{ is defined over } k \mbox{ and there exists a coset representation to } {\bf G}/{\bf H} \}. \]
This is a non-empty collection as it contains ${\bf G}$.
By the Neotherian property, this collection contains a minimal element.
We choose such a minimal element ${\bf H}_0$
and fix corresponding
$\phi_0:Y \to ({\bf G}/{\bf H}_0)(k)$.
We argue to show that this coset representation is the required initial object.

Fix any algebraic representation of $Y$, ${\bf V}$.
It is clear that, if exists, a morphism of algebraic representations from ${\bf G}/{\bf H}_0$ to ${\bf V}$ is unique, as two different ${\bf G}$-maps
${\bf G}/{\bf H}_0\to {\bf V}$ agree nowhere.
We are left to show existence.
To this end we consider
the product representation ${\bf V}\times {\bf G}/{\bf H}_0$ given by $\phi=\phi_{\bf V}\times \phi_0$.
Applying Lemma~\ref{coset contraction} to this product representation we obtain the commutative diagram

%\medskip

\begin{equation} \label{diag-AG}
\xymatrix{ Y \ar@{.>}[r] \ar[d]^{\phi_{\bf V}} \ar[rd]^{\phi} \ar@/^3pc/[rrd]^{\phi_0} & {\bf G}/{\bf H} \ar@{.>}[d]_{i} &  \\
		   {\bf V} & {\bf V}\times {\bf G}/{\bf H}_0 \ar[r]^{~~p_2} \ar[l]_{p_1~~~~} & {\bf G}/{\bf H}_0  }
\end{equation}

By the minimality of ${\bf H}_0$, the ${\bf G}$-morphism $p_2\circ i:{\bf G}/{\bf H} \to {\bf G}/{\bf H}_0$ must be a $k$-isomorphism.
We thus obtain the $k$-${\bf G}$-morphism
\[ p_1\circ i \circ (p_2\circ i)^{-1}:{\bf G}/{\bf H}_0 \to {\bf V}. \]
\end{proof}

%We end this section showing that in several situations the category of algebraic representations is trivial.

\begin{defn} \label{c-ergodic}
We say that the $T$-action on $Y$ is $c$-ergodic if  for every representation ${\bf V}$,
$\phi_{\bf V}$ is essentially constant and its essential image is a ${\bf G}$-fixed point.
Equivalently, the action is $c$-ergodic if the initial coset representation alluded to in Theorem~\ref{a-initial}
is the one point set ${\bf G}/{\bf G}$.
\end{defn}

\begin{prop} \label{c-ergodic-hull}
The $T$-action on $Y$ is $c$-ergodic iff the cocycle $c$ is not cohomologous to a cocycle taking values in a proper algebraic subgroup of $G$.
\end{prop}

\begin{proof}
By Proposition~\ref{proper-cohom}.
\end{proof}

%\begin{prop}
%Assume $Y=T/\Gamma$ for some closed subgroup $\Gamma<T$.
%Let $\kappa:T\times T/\Gamma\to \Gamma$ be the standard cocycle defined in Example~\ref{stand}.
%and let $\delta:\Gamma\to G$ be a continuous homomorphism.
%Assume $c=\delta\kappa$.
%Assume further that $\overline{\delta(T)}^Z={\bf G}$.
%Then the $T$-action on $Y$ is $c$-ergodic.
%\end{prop}
%
%\begin{proof}
%By Proposition~\ref{Gamma-c-ergodic}.
%\end{proof}

\begin{prop} \label{mixing}
Assume that
$Y^j$ is $T$-ergodic for every $j\in \bbN$
and that $c$ is independent of $Y$, that is
there exists
a homomorphism $d:T\to G$
such that for every $t\in T$ and a.e.\ $y\in Y$, $c(t,y)=d(t)$.
Assume further that
$\overline{d(T)}^Z={\bf G}$.
Then the action of $T$ on $Y$ is $c$-ergodic.
\end{prop}

\begin{proof}
We let ${\bf H}<{\bf G}$, $i:{\bf G}/{\bf H}\to {\bf V}$ and $\phi_{ {\bf G}/{\bf H}}:Y\to {\bf G}/{\bf H}(k)$
be as guaranteed by Proposition~\ref{coset contraction}.

We claim that $\dim({\bf G}/{\bf H})=0$.
We will show this arguing by contradiction.
Assume $\dim({\bf G}/{\bf H})>0$.
Then there exists $j\in \bbN$ for which $\dim(({\bf G}/{\bf H})^j)>\dim({\bf G})$.
We fix such $j$.
We consider the map $\phi_{ {\bf G}/{\bf H}}^j:Y^j\to ({\bf G}/{\bf H})^j(k)$
and the measure $\nu=(\phi_{ {\bf G}/{\bf H}}^j)_*\mu^j$ on $({\bf G}/{\bf H})^j(k)$.
Its support is an $d(T)^j$-invariant set.
But $\overline{d(T)^j}^Z={\bf G}^j$ thus $\overline{\supp(\nu)}^Z=({\bf G}/{\bf H})^j$.
On the other hand, $\nu$ is ergodic with respect to the diagonal subgroup in $d(T)^j$, thus, again by Proposition~\ref{coset contraction}, it is
supported on a single orbit of the diagonal group in ${\bf G}^j$.
This orbit must then be Zariski-dense, hence open (as it is locally closed), thus of the same dimension as $({\bf G}/{\bf H})^j$
which is, by the choice of $j$, bigger than $\dim({\bf G})$.
This is an absurd, thus indeed  $\dim({\bf G}/{\bf H})=0$.

It follows that $\phi_{\bf V}$ has a finite image in $V$.
$\phi_{\bf V}$ is than essentially constant,
otherwise, denoting by $\Delta$ the diagonal in $V^2$, $\chi_\Delta\circ \phi_{\bf V}$ gives
a non-constant $T$-invariant function on $Y^2$, contradicting its ergodicity.
The essential image of $\phi_{\bf V}$ is $d(T)$-fixed, hence also  ${\bf G}$-fixed, as ${\bf G}=\overline{d(T)}^Z$.
\end{proof}

\section{Yoneda lemma} \label{s-yoneda}

Throughout this section we fix $T,Y,k,{\bf G}$ and $c$ as defined and discussed in \S\ref{alg-rep}.
A fancy way of thinking of Theorem~\ref{a-initial} is saying that a certain functor is representable by  a coset space.
Taking this formal view point we could take advantage of Yoneda lemma.
We will make this statement precise.
Consider the category $\mathcal{C}$ whose objects are $k$-${\bf G}$-varieties
and a morphism from ${\bf U}$ to ${\bf V}$
is given by a $k$-${\bf G}$-algebraic map $\psi:{\bf U}\to {\bf V}$.
It is easy to see that the application ${\bf V}\mapsto F({\bf V})=L^0(Y,{\bf V}(k))^T$ is functorial.
$F:\mathcal{C} \to {\bf Sets}$ is the functor associating with any object of $\mathcal{C}$ the set of algebraic representations of $Y$
targeting at that object.
The category of algebraic representations of $Y$ is clearly isomorphic with the category of elements of the functor $F$.
In particular, the former category has an initial object if and only if $F$ is representable,
and in this case, by Yoneda lemma, $\Aut(F)$ is isomorphic to the automorphism group in $\mathcal{C}$ of such an initial object.

If $T'$ is a group acting on $Y$ commuting with the $T$ action and the cocycle $c$ extends to a cocycle $T\times T'\times Y\to G$,
then the $T$ action on $L^0(Y,{\bf V}(k))$ extends to a $T\times T'$ action, and in particular $T'$ acts on the set $L^0(Y,{\bf V}(k))^T$.
It follows that $T'$ acts by automorphisms on the functor $F$.
This reasoning gives the homomorphism $d$ considered in the theorem below,
but we will prove it without forcing this language.

\begin{theorem} \label{yoneda}
Let ${\bf H}<{\bf G}$ be a $k$-subgroup and $\phi:Y\to {\bf G}/{\bf H}(k)$ be a $c$-equivariant map
which is
an initial object in the category of algebraic representations of $Y$.
Let $T'$ be a group acting on the Lebesgue space $Y$ commuting with the $T$-action and assume
the cocycle $c$ extends to a cocycle $T\times T'\times Y\to G$.
Then there exists a group homomorphism $d:T'\to N_{\bf G}({\bf H})/{\bf H}(k)$
such that for every $t'\in T'$, a.e.\ $y\in Y$, $t'\phi(y)=\phi(y)d(t')$.
\end{theorem}

Before proving the theorem, let us state without a proof the following proposition
which
provides an identification of $\Aut_{\mathcal{C}}({\bf G}/{\bf H})$ that we keep using throughout the paper.
The proposition is well known and easy to prove.

\begin{prop} \label{aut-identification}
Fix a $k$-subgroup ${\bf H}<{\bf G}$ and denote ${\bf N}=N_{\bf G}({\bf H})$.
This is again a $k$-subgroup.
Any element $n\in {\bf N}$ gives a ${\bf G}$-automorphism of ${\bf G}/{\bf H}$ by
$g{\bf H}\mapsto gn^{-1}{\bf H}$.
The homomorphism $ {\bf N} \to \Aut_{\bf G}({\bf G}/{\bf H})$ thus obtained is onto and its kernel is ${\bf H}$.
Under the obtained identification ${\bf N}/{\bf H} \simeq \Aut_{\bf G}({\bf G}/{\bf H})$,
the $k$-points of the $k$-group ${\bf N}/{\bf H}$ are identified with the $k$-${\bf G}$-automorphisms of ${\bf G}/{\bf H}$,
that is ${\bf N}/{\bf H}(k)\simeq \Aut_\mathcal{C}({\bf G}/{\bf H})$.
\end{prop}

\begin{proof}[Proof of Theorem~\ref{yoneda}]
Let ${\bf H}<{\bf G}$ be a $k$-subgroup and $\phi:Y\to {\bf G}/{\bf H}(k)$ be a $c$-equivariant map
which is
an initial object in the category of algebraic representations of $Y$.
The space $L^0(Y ,{\bf G}/{\bf H}(k))^{T}$
is $T'$ invariant.
For $t'\in T'$ we consider the $T$-representation $(t')^{-1}\phi\in  L^0(Y ,{\bf G}/{\bf H}(k))^T$.
By the fact that $\phi$ forms an initial object we get the dashed vertical arrow, which we denote $d(t')$, in the following diagram.
\begin{equation} \label{a-diag-AG}
\xymatrix{ Y \ar[r]^{\phi} \ar[rd]_{(t')^{-1}\phi} & {\bf G}/{\bf H} \ar@{.>}[d]^{d(t')}  \\
		    & {\bf G}/{\bf H}  }
\end{equation}
By the uniqueness of the dashed arrow, the correspondence $t'\mapsto d(t')$ is easily checked to form a homomorphism from $T'$ to the group
of $k$-${\bf G}$-automorphism of $ {\bf G}/{\bf H}$,
which we identify with $N_{\bf G}({\bf H})/{\bf H}(k)$
using Proposition~\ref{aut-identification}.
By viewing $d(t')$ as an element of $N_{\bf G}({\bf H})/{\bf H}(k)$
we obtain for every $t'\in T'$, a.e.\ $y\in Y$, $t'\phi(y)=\phi(y)d(t')$.
\end{proof}

\begin{prop}[cf.\ {\cite[Proposition 3.8]{MS-SR}, \cite[Proposition 4.2]{FM}}] \label{prop:prod}
%Assume $Y$ is $T$-ergodic.
Let $T'$ be an lcsc group acting on $Y$ commuting with the $T$ action and assume the cocycle $c$ extends to a cocycle
$\tilde{c}:T\times T'\times Y\to G$.
Assume the algebraic group ${\bf G}$ is an adjoint form simple group.
Assume that the $(T\times T')$-action on $Y$ is $\tilde{c}$-ergodic, but the $T$-action on $Y$ is not $c$-ergodic (see Definition~\ref{c-ergodic}).
Then the cocycle $\tilde{c}$ is cohomologous to a homomorphism of the form $d\circ\pi_2$ for some continuous homomorphism $d:T'\to G$.
\end{prop}

\begin{proof}
Let ${\bf H}<{\bf G}$ be a $k$-subgroup and $\phi:Y\to {\bf G}/{\bf H}(k)$ be a $c$-equivariant map
which is
an initial object in the category of algebraic representations of the $T$-Lebesgue space $Y$,
as guaranteed by Theorem~\ref{a-initial}.
By Theorem~\ref{yoneda} we get a homomorphism $d:T'\to N_{\bf G}({\bf H})/{\bf H}(k)$ satisfying for every $t'\in T'$, a.e.\ $y\in Y$, $t'\phi(y)=\phi(y)d(t')$.

We claim that ${\bf H}=\{e\}$. Assume not.
By the assumptions that the $T$-action on $Y$ is not $c$-ergodic, ${\bf H}\lneq {\bf G}$.
If also ${\bf H}\neq\{e\}$, by the simplicity of ${\bf G}$, we conclude that ${\bf N}=N_{\bf G}({\bf H})\lneq {\bf G}$.
We let $\psi$ be the post-composition of $\phi$ with the natural $k$-${\bf G}$-map ${\bf G}/{\bf H}\to {\bf G}/{\bf N}$.
Since ${\bf G}/{\bf H}\to {\bf G}/{\bf N}$ is $ N_{\bf G}({\bf H})/{\bf H}(k)$-invariant, we get that $\psi$ is $T'$-invariant.
Thus $\psi\in L^0(Y, {\bf G}/{\bf N}(k))^{T\times T'}$.
This contradicts the assumption that the $(T\times T')$-action on $Y$ is $\tilde{c}$-ergodic.
Thus ${\bf H}=\{e\}$.

We now claim that the map $\phi:Y\to G$ conjugates $\tilde{c}$ to $d\circ \pi_2$.
Indeed, for $(t,t')\in T\times T'$, for a.e\ $y\in Y$ we have
\[ \tilde{c}(tt',y)\phi(y)=
\tilde{c}(t',ty)\tilde{c}(t,y)\phi(y)=
\tilde{c}(t',ty)\phi(ty)=
\]\[
\tilde{c}(t',t'^{-1}(tt')y)\phi(t'^{-1}(tt')y)=
t'\phi(tt'y)=
\phi(tt'y)d(t'). \]
Finally, by the fact that $\tilde{c}$ and $\phi$ are measurable, we note that $d$ is measurable, hence continuous.
\end{proof}

\section{Algebraic representations and amenable actions} \label{s:amen}

In this section we continue the study of algebraic representations of ergodic actions.
Our basic tool is the following theorem which is proven together with Bruno Duchesne and Jean L\'{e}cureux in \cite{BDL}.

\begin{theorem} \label{BDL}
Let $R$ be a locally compact group and $Y$ an ergodic, amenable $R$ Lebesgue space.
Let $(k,|\cdot|)$ be a field endowed with an absolute value.
Assume that as a metric space $k$ is complete and separable.
Let ${\bf G}$ be an adjoint form simple $k$-algebraic group.
Let $f:R\times Y \to {\bf G}(k)$ be a measurable cocycle.

Then either
there exists
a $k$-algebraic subgroup ${\bf H}\lneq {\bf G}$
and an $f$-equivariant measurable map $\phi:Y\to {\bf G}/{\bf H}(k)$,
or there exists
a complete and separable metric space $V$ on which $G$ acts by isometries
with bounded stabilizers
and an $f$-equivariant measurable map $\phi':Y\to V$.

Furthermore,
in case the image $|\cdot|:k\to [0,\infty)$ is closed,
the space $V$ could be chosen such that the quotient topology on $V/G$ is Hausdorff,
in case $k$ is a local field the $G$-action on $V$ is proper
and in case $k=\bbR$ and $G$ is non-compact the first alternative always occurs.
\end{theorem}

Note that the notion of a proper action is defined and discussed in \S\ref{proper prelim}.

\begin{remark}
\begin{itemize}
\item
For a local field of 0 characteristic
the above result is well known and follows from
\cite[Corollaries 3.2.17 and 3.2.19]{zimmer-book}.
\item
The properness of the action in the local field case is automatic,
see \cite{BG}.
\item
The special feature of the reals that plays a part here is that compact subgroups of algebraic real groups are always algebraic subgroups.
\end{itemize}
\end{remark}

\begin{theorem} \label{nontrivial}
Let $T$ be a lcsc group and $X$ be a $T$-Lebesgue space.
Assume that the $T$ action on $X$ is probability measure preserving.
Let $(k,|\cdot|)$ be a field endowed with an absolute value.
Assume that as a metric space $k$ is complete and separable.
Let ${\bf G}$ be an adjoint form simple $k$-algebraic group and denote $G={\bf G}(k)$.
Let $c:T\times X \to G$ be a measurable cocycle.
Assume $c$ is not cohomologous to a cocycle taking values in a bounded subgroup of $G$.

Let $B$ be an amenable and metrically ergodic $T$-Lebesgue space
and let $f:T\times B\times X\to G$ be defined by $f(t,b,x)=c(t,x)$.
Assume either that $X$ is $T$-transitive or that  the image of $|\cdot|:k\to [0,\infty)$ is closed.
Then the $T$-action on $B\times X$ is not $f$-ergodic (see Definition~\ref{c-ergodic}).
\end{theorem}

\begin{proof}
We will prove the theorem by contradiction.
Assuming
that the $T$-action on $B\times X$ is $f$-ergodic, we will show that the cocycle $c$
is cohomologous to a cocycle taking values in a bounded subgroup of $G$.

Note that by Lemma~\ref{me-e}, $B\times X$ is $T$-ergodic.
By taking $Y=B\times X$ and $R=T$ in Theorem~\ref{BDL},
there exists
a complete and separable metric space $V$ on which $G$ acts by isometries
with bounded stabilizers
and an $f$-equivariant measurable map $\phi:Y\to V$.
Denoting the $G$-invariant metric on $V$ by $d$, we define a metric $D$ on $L^0(X,V)$ by setting for $\alpha,\beta\in L^0(X,V)$,
\begin{equation} \label{metric}
D(\alpha,\beta)=\int_X \min\{d(\alpha(x),\beta(x)),1\}.
\end{equation}
It is easy to check that $D$ is a $T$-invariant metric and that the map $\Phi:B\to L^0(X,V)$ defined by $\Phi(b)(x)=\phi(b,x)$
(using Fubini theorem)
is $T$-equivariant.
By the metric ergodicity of $B$ we conclude that $\Phi$ is essentially constant, $\Phi(B)=\{\psi\}$
for some $\psi\in L^0(X,V)$.
It is easy to check that $\psi \in L^0(X,V)^T$.
If $X$ is $T$-transitive then the image of $\psi$ is $G$-transitive.
Also, if the image $|\cdot|:k\to [0,\infty)$ is closed,
by the fact that $V/G$ is Hausdorff, using Proposition~\ref{coset contraction},
we conclude that the image of $\psi$ is essentially transitive.
In both cases we may view $\psi$ as a function taking values in $G/K$, where $K$ is a stabilizer of a point in $V$, hence bounded.
That is, $\psi \in L^0(X,G/K)^T$.
By Proposition~\ref{proper-cohom}, $c$ is cohomologous to a cocycle taking values in $K$ which is bounded in of $G$.
This contradicts our assumptions on $c$.
\end{proof}

\section{Super-rigidity for products} \label{s:products}

For $S=S_1\times S_2\times \cdots\times S_n$, a product of groups,
and for $i\leq n$, we denote by $\pi_i:S\to S_i$ the natural factor map.
We use the notation $\hat{S}_i=\Ker(\pi_i)$, thus $S\simeq \hat{S}_i\times S_i$.
An $S$-Lebesgue space $X$ is called irreducible is for each $i\leq n$,
$X$ is $\hat{S}_i$-ergodic.
A subgroup $\Gamma<S$ is called irreducible if $S/\Gamma$ is irreducible as an $S$-Lebesgue space.
Equivalently, $\Gamma$ is irreducible if $S_i=\overline{\pi_i(\Gamma)}$ for each $i$.
The following theorem generalizes and improves \cite[Theorem~C]{FM}.

\begin{theorem}[Cocycle Super-Rigidity for products] \label{prod-cocycle}
Let $S=S_1\times S_2\times \cdots\times S_n$ be a product of lcsc groups,
and let $X$ be an irreducible pmp $S$-Lebesgue space.

Let $k$ be a local field.
Let $G$ be the $k$-points of an adjoint form simple algebraic group defined over $k$.
Let $c:S\times X \to G$ be a measurable cocycle.
Assume that $c$ is not cohomologous to a cocycle taking values in a proper algebraic subgroup or a bounded subgroup of $G$.
Then there exists $i \leq n$ and a continuous homomorphism $d:S_i\to G$
such that $c$ is cohomologous to the homomorphism $d\circ\pi_i$.
\end{theorem}

%%%%%%%%%%%%%

\begin{proof}
We will assume by contradiction that
there exists no $i \leq n$ and a continuous homomorphism $d:S_i\to G$
such that $c$ is cohomologous to the homomorphism $d\circ\pi_i$.

Theorem~\ref{boundary} guarantees that for every $i\leq n$
there exists an amenable and metrically ergodic $S_i$-space $B_i$,
which we now fix.
We let $c_0=c$
and for $0< j\leq n$ we let $c_j:S\times X\times \prod_{i=1}^j B_i\to G$ be the obvious extension of $c$.
Note that $\hat{S}_{j+1}$ acts ergodically on $X\times \prod_{i=1}^j B_i$ (and actually on $X\times \prod_{i=1}^j B_i\times \prod_{i=j+2}^n B_i$)
by Lemma~\ref{me-e}, as $X$ is $\hat{S}_{j+1}$-ergodic probability measure preserving (by the irreducibility assumption)
and by Proposition~\ref{product-stable} applied to the spaces $B_i$.
We claim that for every $0\leq j\leq n-1$,
the $\hat{S}_{j+1}$-action on $X\times \prod_{i=1}^j B_i$ is $c_j$-ergodic.

The case $j=0$ follows by applying
Proposition~\ref{prop:prod} to
$T=\hat{S}_{1}$, $T'=S_{1}$ and $Y= X$,
as by assumption $c$ is not cohomologous to a cocycle taking values in a proper algebraic subgroup, 
hence by Proposition~\ref{c-ergodic-hull} the $S$-action on $X$ is $c$-ergodic.

Assume that for some $1\leq j\leq n-1$ the claim is true for $j-1$, but not for $j$.
Setting $T=\hat{S}_{j+1}$, $T'=S_{j+1}$ and $Y= X\times \prod_{i=1}^{j} B_i$,
we get by Proposition~\ref{prop:prod} that
the $S$ action on $Y$ is not $c_j$-ergodic.
That is, there exists
a proper $k$-algebraic subgroup ${\bf H}\lneq {\bf G}$
and a measurable map $\phi:X\times \prod_{i=1}^{j} B_i\to {\bf G}/{\bf H}(k)$
satisfying for every $s\in S$ and a.e.\ $y=(x,b_1,\ldots,b_j)\in X\times \prod_{i=1}^{j} B_i$, $\phi(sy)=c(s,x)\phi(y)$.
Therefore, for $s\in \hat{S}_j$ and for a generic $b_j\in B_j$, the map
$\phi_{b_j}:X\times \prod_{i=1}^{j-1} B_i\to {\bf G}/{\bf H}(k)$ defined by
$\phi_{b_j}(x,b_1,\ldots,b_{j-1})=\phi(x,b_1,\ldots,b_j)$
satisfies $\phi_{b_j}(x,b_1,\ldots,b_{j-1})=c(s,x)\phi_{b_j}(x,b_1,\ldots,b_{j-1})$.
This contradicts the assumption that  the claim is true for $j-1$.

Thus the claim is proven by induction. In particular, we conclude that the claim holds for $j=n-1$.
But this contradicts Theorem~\ref{nontrivial}, applied to $T=\hat{S}_n$, $B=\prod_{i=1}^{n-1} B_i$ and $f=c_{n-1}$,
as by Proposition~\ref{product-stable} the action of $\hat{S}_n$ on $\prod_{i=1}^{n-1} B_i$ is amenable and metrically ergodic.
Therefore, by contradiction, we deduce that there exists an $i \leq n$ and a continuous homomorphism $d:T_i\to G$
such that $c$ is cohomologous to the homomorphism $d\circ\pi_i$.
\end{proof}

\begin{proof}[Proof of Theorem~\ref{thm:lattice}]
First observe that there exists a subfield $k'<k$ which is separable and complete, ${\bf G}$ is defined over $k'$ and $\delta(\Gamma)<{\bf G}(k')$.
Indeed, take $k'$ to be the closure of a finite extension of the subfield generated by $\delta(\Gamma)$ over which ${\bf G}$ is defined.
Replacing $k$ by $k'$ we may assume $k$ is separable.
We do so.

We let $\kappa$ be as in Example~\ref{stand} and consider the cocycle $c=\delta\kappa$.
By Proposition~\ref{Gamma-c-ergodic},
$c$ is not cohomologous to a cocycle taking values in a proper algebraic subgroup or a bounded subgroup.
It follows by Theorem~\ref{prod-cocycle} that
there exists $i \leq n$ and a continuous homomorphism $d:S_i\to G$
such that $c$ is cohomologous to the homomorphism $d\circ\pi_i$.
By Proposition~\ref{lattice-vs-cocycle} $\delta$ extends to a homomorphism $T\to G$.
By the simplicity of ${\bf G}$ it is easy to see that such a homomorphism must factor through the projection $\pi_i:T\to T_i$ for some $i \leq n$.
\end{proof}

\section{Bi-algebraic representations of bi-actions} \label{algebraic bi-actions}

In this section we study a generalization of the setting studied in \S\ref{alg-rep}.
Throughout this section we fix
\begin{itemize}
\item lcsc groups $S$ and $T$,
\item an $S\times T$-Lebesgue space $Y$,
\item a field $k$ with an absolute value which is separable and complete (as a metric space),
\item a $k$-algebraic group ${\bf G}$,
\item a measurable cocycle
$c:S \times Y\ec T \to G$,
where $G={\bf G}(k)$ is regarded as a Polish group (Proposition~\ref{polishing}).
We denote by $\tilde{c}$ the pullback of $c$ to $Y$.
\end{itemize}

\begin{defn}
Given all the data above, a bi-algebraic representation of $Y$
consists of the following data:
\begin{itemize}
\item a $k$-algebraic group ${\bf L}$.
\item a $k$-$({\bf G}\times {\bf L})$-algebraic variety ${\bf V}$,
\item a homomorphism $d:T\to {\bf L}(k)$ with a Zariski dense image,
\item a measurable map $\phi:Y \to {\bf V}(k)$ such that for every $s\in S$, $t\in T$ and almost every $y\in Y$,
\[ \phi(sty)=\tilde{c}(s,y)d(t)\phi(y). \]
\end{itemize}
We abbreviate the notation by saying that ${\bf V}$ is a bi-algebraic representation of $Y$,
denoting the extra data by ${\bf L}_{\bf V}, d_{\bf V}$ and $\phi_{\bf V}$.
Given another bi-algebraic representation ${\bf U}$ we let ${\bf L}_{{\bf U},{\bf V}}<{\bf L}_{\bf U}\times {\bf L}_{\bf V}$ be the Zariski closure of the image of $d_{\bf U}\times d_{\bf V}:T\to {\bf L}_{\bf U}\times {\bf L}_{\bf V}$. ${\bf L}_{{\bf U},{\bf V}}$ acts on ${\bf U}$ and ${\bf V}$ via its projections on ${\bf L}_{\bf U}$ and ${\bf L}_{\bf V}$ correspondingly.
A morphism of bi-algebraic representations of $Y$ from the bi-algebraic representation ${\bf U}$ to the bi-algebraic representation ${\bf V}$ is
a $k$-algebraic map $\psi:{\bf U}\to {\bf V}$ which is ${\bf G}\times{\bf L}_{{\bf U},{\bf V}}$-equivariant,
and such that $\phi_{\bf V}=\psi\circ \phi_{\bf U}$.

A bi-algebraic representation ${\bf V}$ of $Y$ is said to be a coset bi-algebraic representation if
${\bf V}={\bf G}/{\bf H}$ for some $k$-algebraic subgroup ${\bf H}<{\bf G}$,
and ${\bf L}$ is a $k$-subgroup of $N_{\bf G}({\bf H})/{\bf H}$ which acts on ${\bf V}$ as described in Proposition~\ref{aut-identification}.
\end{defn}

It is clear that the collection of bi-algebraic representations of $Y$ and their morphisms form a category.
In Theorem~\ref{a-initial} we proved the existence of initial objects in categories of algebraic representations of ergodic actions.
Our notion of bi-algebraic representation is slightly harder to handle, as the algebraic group ${\bf L}$ in its definition is arbitrary.
Due to Proposition~\ref{mixing} this obstacle could be overcome, under strong enough ergodic assumptions.

\begin{theorem} \label{initial}
Assume $(Y \ec S)^j$ is $T$-ergodic for every $j\in \bbN$.
Then the category of bi-algebraic representations of $Y$ has an initial object.
Moreover, there exists an initial object which is a coset bi-algebraic representation.
\end{theorem}

We will first prove the following lemma which is a strengthening of
Proposition~\ref{coset contraction}.

\begin{lemma} \label{mixing coset contraction}
Assume $(Y \ec S)^j$ is $T$ ergodic for every $j\in \bbN$.
Let ${\bf V}$ be a bi-algebraic representation of $Y$.
Then there exists a coset bi-algebraic representation of $Y$ for some $k$-algebraic subgroup ${\bf H}<{\bf G}$
and a morphism of bi-algebraic representations $i:{\bf G}/{\bf H}\to {\bf V}$.
\end{lemma}

\begin{proof}
Note that $Y$ is an ergodic $(S\times T)$-Lebesgue and the map $S\times T \times Y\to ({\bf G}\times {\bf L})(k)$
defined by $(s,t,y)\mapsto (\tilde{c}(s,y),d(t))$
is an $(S\times T)$-cocycle from $Y$ to $({\bf G}\times {\bf L})$,
as $\tilde{c}$ is a pull back of $c$, hence independent of the $T$ action, and $d$ depends only on the group $T$.
Considering ${\bf V}$ as an algebraic representation
of the ergodic $(S\times T)$-Lebesgue space $Y$ via that cocycle,
by applying Proposition~\ref{coset contraction}, we may reduce to the case
${\bf V}=({\bf G}\times {\bf L})/{\bf M}$ for some $k$-algebraic subgroup ${\bf M}<{\bf G}\times {\bf L}$.
We do so.
Denote the obvious projection from ${\bf G} \times {\bf L}$
to ${\bf G}$ and ${\bf L}$ correspondingly by $\pi_1$ and $\pi_2$.
The composition of the map $\phi:Y \to ({\bf G}\times {\bf L})/{\bf M}(k)$ with the $G$-invariant map $({\bf G}\times {\bf L})/{\bf M}(k)\to {\bf L}/\pi_2({\bf M})(k)$
clearly factors through $Y\ec S$, and thus gives a $d$-equivariant map from the $T$-Lebesgue space $Y\ec S$ to ${\bf L}/\pi_2({\bf M})(k)$.
Thus ${\bf L}/\pi_2({\bf M})$
becomes a coset algebraic representation of
the $T$-Lebesgue space
$Y\ec S$
via the cocycle $d$
to the algebraic group ${\bf L}$.
Applying Proposition~\ref{mixing}, we conclude that $\pi_2({\bf M})={\bf L}$.
It follows that as ${\bf G}$-varieties, $ ({\bf G}\times {\bf L})/{\bf M} \simeq {\bf G}/\pi_1({\bf M})$.
The lemma follows easily.
\end{proof}

The proof of the theorem proceeds as the proof of Theorem~\ref{a-initial}.

\begin{proof}[Proof of Theorem~\ref{initial}]
We consider the collection
\[ \{{\bf H}<{\bf G}~|~{\bf H}\mbox{ is defined over } k \mbox{ and there exists a coset bi-representation to } {\bf G}/{\bf H} \}. \]
This is a non-empty collection as it contains ${\bf G}$.
By the Neotherian property, this collection contains a minimal element.
We choose such a minimal element ${\bf H}_0$
and fix corresponding algebraic $k$-subgroup ${\bf L}_0<N_{\bf G}({\bf H}_0)/{\bf H}_0)$,
homomorphism $d_0:T \to {\bf L}_0(k)$
and representation $\phi_0:Y \to ({\bf G}/{\bf H}_0)(k)$.
We argue to show that this coset bi-representation is the required initial object.

Fix any bi-algebraic representation of $Y$, ${\bf V}$.
It is clear that, if exists, a morphism of bi-algebraic representations from ${\bf G}/{\bf H}_0$ to ${\bf V}$ is unique, as two different ${\bf G}$-maps
${\bf G}/{\bf H}_0\to {\bf V}$ agree nowhere.
We are left to show existence.
To this end we consider
the product bi-representation ${\bf V}\times {\bf G}/{\bf H}_0$ given by the data $\phi=\phi_{\bf V}\times \phi_0$,
$d=d_{\bf V}\times d_0$ and ${\bf L}$ being the Zariski closure of $d(T)$ in ${\bf L}_{\bf V}\times {\bf L}_0$.
Applying Lemma~\ref{mixing coset contraction} to this product bi-representation we obtain the commutative diagram

%\medskip

\begin{equation} \label{diag-AG}
\xymatrix{ Y \ar@{.>}[r] \ar[d]^{\phi_{\bf V}} \ar[rd]^{\phi} \ar@/^3pc/[rrd]^{\phi_0} & {\bf G}/{\bf H} \ar@{.>}[d]_{i} &  \\
		   {\bf V} & {\bf V}\times {\bf G}/{\bf H}_0 \ar[r]^{~~p_2} \ar[l]_{p_1~~~~} & {\bf G}/{\bf H}_0  }
\end{equation}

By the minimality of ${\bf H}_0$, the ${\bf G}$-morphism $p_2\circ i:{\bf G}/{\bf H} \to {\bf G}/{\bf H}_0$ must be a $k$-isomorphism,
hence an isomorphism of bi-algebraic representations.
We thus obtain the morphism of bi-algebraic representations
\[ p_1\circ i \circ (p_2\circ i)^{-1}:{\bf G}/{\bf H}_0 \to {\bf V}. \]
\end{proof}

\begin{defn}
We say that the $(S,T)$ bi-action on $Y$ is $c$-ergodic if  for every representation ${\bf V}$,
$\phi_{\bf V}$ is essentially constant and its essential image is a ${\bf G}$-fixed point.
Equivalently, the bi-action is $c$-ergodic if the initial coset bi-representation alluded to in Theorem~\ref{initial}
is the one point set ${\bf G}/{\bf G}$.
\end{defn}

As in the discussion carried at the beginning of \S\ref{s-yoneda}, Theorem~\ref{initial} could be interpreted as saying that a certain functor is representable.
Consider the category $\mathcal{D}$ whose objects are $k$-${\bf G}$-varieties ${\bf V}$ endowed with
a choice of a $k$-algebraic subgroup ${\bf L}_{\bf V}<\Aut_{\bf G}({\bf V})$, where a morphism from ${\bf U}$ to ${\bf V}$
is given by a $k$-algebraic homomorphism $f:{\bf L}_{\bf U} \to {\bf L}_{\bf V}$ and a $k$-algebraic map $\psi:{\bf U}\to {\bf V}$
which is ${\bf G}$ and $f$-equivariant.
We have a natural functor $F$ from $\mathcal{D}$ to ${\bf Sets}$, given by associating with any object of $\mathcal{D}$ all the bi-algebraic representations of $Y$
targeting at that object.
The category of bi-algebraic representations of $Y$ is clearly isomorphic with the category of elements of the functor $F$.
In particular, the former category has an initial object if and only if $F$ is representable,
and in this case, by Yoneda lemma, $\Aut(F)$ is isomorphic to the automorphism group in $\mathcal{D}$ of such an initial object.
Using this point of view, one may care to formulate an exact analogue of Theorem~\ref{yoneda}.
We will not do it here, as all we need is the following proposition.

\begin{prop} \label{T'}
Let $T'$ be a group acting on $Y$ commuting with $S$ and $T$.
Assume the cocycle $\tilde{c}$ is $T'$-invariant in the sense that
for all $s\in S$, $t'\in T'$ and a.e.\ $y\in Y$, $\tilde{c}(s,ty)=\tilde{c}(s,y)$.
Assume $(Y \ec S)^j$ is both $T$ and $T'$ ergodic for every $j\in \bbN$.
Let ${\bf G}/{\bf H}$, ${\bf L}<N_{\bf G}({\bf H})/{\bf H}$,
$d:T \to {\bf L}(k)$ and $\phi:Y \to ({\bf G}/{\bf H})(k)$ be an initial object
in the category of bi-algebraic representations of the $(S,T)$ bi-space $Y$, as guaranteed by Theorem~\ref{initial}.
Then there exists a $k$-algebraic group
${\bf L}'<N_{\bf G}({\bf H})/{\bf H}$
which commutes with ${\bf L}$
and a homomorphism $d':T'\to {\bf L}'(k)$
such that ${\bf G}/{\bf H}$, ${\bf L}'$, $d'$ and $\phi$ form an initial object in the category of bi-algebraic representations of the $(S,T')$ bi-space $Y$.
\end{prop}

\begin{proof}
Using the facts that $T'$ commutes with $S$ and $T$ and that $\tilde{c}$ is $T'$-invariant,
it is easy to check that for every $s\in S$, $t\in T$, $t'\in T$, a.e.\ $y\in Y$,
\[ \phi(t'sty)=\tilde{c}(s,y)d(t)\phi(t'y). \]
Equivalently, for every $t'\in T'$ the data given by the algebraic group ${\bf L}$,
the $k$-$({\bf G}\times {\bf L})$-algebraic variety ${\bf G}/{\bf H}$,
the homomorphism $d:T\to {\bf L}(k)$ and $\phi\circ t'$
forms a bi-algebraic representation of the $(S,T)$ bi-space $Y$.

By the fact that the bi-algebraic representation given by ${\bf L}$, ${\bf G}/{\bf H}$,
$d$ and $\phi$ forms an initial object we get the dashed vertical arrow, which we denote $d'(t')$, in the following diagram.
\begin{equation} \label{diag-AG}
\xymatrix{ Y \ar[r]^{\phi} \ar[rd]_{\phi\circ t'} & {\bf G}/{\bf H} \ar@{.>}[d]^{d'(t')}  \\
		    & {\bf G}/{\bf H}  }
\end{equation}
By the uniqueness of the dashed arrow, the correspondence $t'\mapsto d'(t')$ is easily checked to form a homomorphism from $T'$ to the group
of $k$-${\bf G}$-automorphism of $ {\bf G}/{\bf H}$,
which we identify with $N_{\bf G}({\bf H})/{\bf H}(k)$
using Proposition~\ref{aut-identification}.
We define ${\bf L}'=\overline{d'(T')}^Z$.
$d'(T')$ commutes with ${\bf L}$ hence so does ${\bf L}'$.
We thus obtain a bi-representation of the $(S,T')$ bi-space $Y$
given by the algebraic group ${\bf L}'$, the variety ${\bf G}/{\bf H}$, the homomorphism $d':T' \to {\bf L}'(k)$ and the (same old) map
$\phi:Y \to {\bf G}/{\bf H}(k)$.
By Theorem~\ref{initial} there exists an initial object in the category of  bi-representation of the $(S,T')$ bi-space $Y$,
which we denote by ${\bf L}_0$, ${\bf G}/{\bf H}_0$,  $d_0$ and $\phi_0$.
Thus there is a $k$-${\bf G}$-morphism from ${\bf G}/{\bf H}_0$ to ${\bf G}/{\bf H}$.

We now note that the assumptions on the groups $T$ and $T'$ are symmetric.
Interchanging the roles of $T$ and $T'$, by the same reasoning as above, there is a
$k$-${\bf G}$-morphism from ${\bf G}/{\bf H}$ to ${\bf G}/{\bf H}_0$.
Thus the two spaces are isomorphic and the proposition follows.
\end{proof}

\section{Higher rank Super-Rigidity} \label{reals}

The following theorem is the technical hart of the paper. It is a general Super-Rigidity result
which is valid for any field with absolute value which is metrically separable and complete,
under the extra assumption of the non-triviality of a certain category of bi-algebraic representations.

\begin{theorem} \label{main alg}
Let $S$ and $T$ be locally compact groups.
Assume $T$ is generated as a topological group by the
subgroups $T_0,T_1,T_2,\ldots$
(for finitely or countably many $T_i$'s)
such that for each $i=1,2,\ldots$ the groups $T_{i-1}$ and $T_i$ commute.
Let $Y$ be an $S\times T$ Lebesgue space.
Assume $(Y \ec S)^j$ is $T_i$-ergodic for every $i=0,1,2,\ldots$ and every $j\in \bbN$.

Let $k$ be a field endowed with an absolute value which is complete and separable as a metric space.
Let ${\bf G}$ be the $k$-points of an adjoint form simple algebraic group defined over $k$
and denote $G={\bf G}(k)$.
Let $c:S\times Y \ec T \to G$ be a measurable cocycle.
Assume the $S$-action on $Y\ec T$ is $c$-ergodic,
but the $(S,T_0)$ bi-action on $Y$ is not $c$-ergodic.
Then there exists a continuous homomorphism $d:T\to G$ and a measurable map
$\phi:Y\to G$ with the following property:
for every $s\in S$, $t\in T$, for almost every $y\in Y$,
\[ \phi(sty)=\tilde{c}(s,y)\phi(y) d(t)^{-1},\]
where $\tilde{c}:S\times Y\to G$ is the pullback of $c$.
\end{theorem}

\begin{proof}
We let
${\bf G}/{\bf H}$, ${\bf L}_0<N_{\bf G}({\bf H})/{\bf H}$,
$d_0:T_0 \to {\bf L}_0(k)$ and $\phi:Y \to ({\bf G}/{\bf H})(k)$ be an initial object
in the category of bi-algebraic representations of the $(S,T_0)$ bi-space $Y$, as guaranteed by Theorem~\ref{initial}.
By the assumption that
the $(S,T_0)$ bi-action on $Y$ is not $c$-ergodic, ${\bf H}\lneq {\bf G}$.
By Proposition~\ref{T'} applied to $T=T_0$ and $T'=T_1$
we get a $k$-algebraic subgroup ${\bf L}_1<N_{\bf G}({\bf H})/{\bf H}$ and a homomorphism
$d_1:T_1\to {\bf L}_1(k)$ such that
${\bf L}_1$, ${\bf G}/{\bf H}$,
$d_1$ and $\phi$ form an initial object
in the category of bi-algebraic representations of the $S,T_1$-space $Y$.

Repeating this argument for each pair of groups $T_{i-1},T_i$ we get
$k$-algebraic subgroups ${\bf L}_i<N_{\bf G}({\bf H})/{\bf H}$ and homomorphisms
$d_i:T_i\to {\bf L}_i(k)$
satisfying for every $t\in T_i$, for almost every $y\in Y$,
$\phi(ty)=d_i(t)\phi(y)$.

Denote ${\bf N}=N_{\bf G}({\bf H})$.
This is a $k$-algebraic subgroup of ${\bf G}$.
%${\bf N}/{\bf H}$ acts on ${\bf G}/{\bf H}$
%where the action of $n{\bf H}\in {\bf N}/{\bf H}$ on ${\bf G}/{\bf H}$ is by
%$g{\bf H} \mapsto gn^{-1}{\bf H}$.
%It is well known and easy to prove that by this ${\bf N}/{\bf H}$ is
%identified with $\Aut_{\bf G}({\bf G}/{\bf H})$.
%We make this identification.
%In particular,
%we identify the groups ${\bf L}_i$ as subgroups of ${\bf N}/{\bf H}$ and
Denote by ${\bf L}$ the algebraic subgroup generated by ${\bf L}_0,\ldots,{\bf L}_n$.
${\bf L}<{\bf N}/{\bf H}$ is a $k$-algebraic subgroup.

We also denote by ${\bf L}'<{\bf N}$ the preimage of ${\bf L}$.
${\bf L}'<{\bf N}$ is a $k$-algebraic subgroup.
We conclude that the $k$-${\bf G}$-morphism $\pi:{\bf G}/{\bf H} \to {\bf G}/ {\bf L}'$
is ${\bf L}_i$ invariant for every $i$.
Since $T$ is topologically generated by the groups $T_i$ and ${\bf L}_i=\overline{d_i(T_i)}^Z$,
it follows that $\pi\circ \phi:Y \to {\bf G}/ {\bf L}'$ factors through
$Y \ec T$ and we get a $c$-equivariant map $\psi:Y \ec T \to  {\bf G}/ {\bf L}'$.
By the assumption that the $S$-action on $Y\ec T$ is $c$-ergodic, we conclude that ${\bf L}'={\bf G}$.
Thus ${\bf G}$ normalizes $ {\bf H}$.
Since ${\bf G}$ is simple and ${\bf H}\lneq {\bf G}$ we conclude that ${\bf H}$ is trivial.
In particular ${\bf L}={\bf L}'={\bf G}$ and it acts on ${\bf G}/{\bf H}={\bf G}$ by right multiplication.

To summarize:
we have homomorphisms $d_i:T_i\to G$ and a measurable map $\phi:Y\to G$
satisfying for every $s\in S$, every $i$, every $t\in T_i$, for almost every $y\in Y$,
$\phi(sty)=\tilde{c}(s,y)\phi(y)d_i(t)^{-1}$.
The algebraic group generated by $d_0(T_0),\ldots, d_n(T_n)$ is ${\bf G}$.
Since the groups $d_i(T_i)$ preserve the support of $\phi(Y)$ we conclude that
the Zariski closure of $\phi(Y)$ is ${\bf G}$ invariant.
It follows that $\phi(Y)$ is Zariski-dense
(this is true for whatever model of $\phi$, that is disregarding any given null set of $Y$).

Consider the polynomial ring $k[{\bf G}]$ and the $k$-algebra $L^0(Y,k)$ consisting
of classes of measurable $k$-valued functions modulo null sets,
and define $\phi^*:k[{\bf G}] \to L^0(Y,k)$ by $\phi^*(p)=p\circ \phi$.
By the fact that $\phi(Y)$ is Zariski dense in ${\bf G}$ we conclude that $\phi^*$
is injective.
$\phi^*$ is $T_i$-equivariant for all $i$, where $T_i$ acts on $k[{\bf G}]$ via the right translation action of $d_i(T_i)$.
The subalgebra $\phi^*(k[{\bf G}])$ is invariant under the various groups $T_i$, hence also under the dense group in $T$ which is generated by them.
By \cite[Corollary 1.9]{borel} $k[{\bf G}]$ is a union of finite dimensional subspaces which are left and right translation invariant.
It follows that $\phi^*(k[{\bf G}])$ is $T$-invariant.
We use the injectivity of $\phi^*$ to define a $T$-action on $k[{\bf G}]$, extending the $T_i$-actions.
The left translation of $G$ gives rise to an action on $k[{\bf G}]$ which commutes with the actions of the groups $T_i$.
It follows that the action of $T$ on $k[{\bf G}]$ commutes with the left translation action of $G$.
As the group of automorphisms of the affine algebra $k[{\bf G}]$ which commute with
left translations is exactly the group of right translations by $G$, we get a homomorphism $d:T\to G$
which clearly extends the homomorphisms $d_i$.
This homomorphism is continuous, since $k[{\bf G}]$ is a union of finite dimensional right $G$ invariant subspaces.
It follows that
\[ \phi(sty)=\tilde{c}(s,y)\phi(y) d(t)^{-1}.\]
\end{proof}

%%%%%%%

%%%%%%%%%%%%

\begin{theorem} \label{secthm}
Let $T$ be a locally compact second countable group.
Assume $T$ is generated as a topological group by the closed, non-compact subgroups $T_0,T_1,T_2,\ldots$
(for any finite number or countably many $T_i$'s)
such that for each $i=1,2,\ldots$ the groups $T_{i-1}$ and $T_i$ commute.
Assume moreover that $T_0$ is amenable and $T/T_0$ admits no non-trivial $T$-invariant semi-metric.
Let $X$ be a $T$-Lebesgue space with a finite invariant mixing measure.

Let $(k,|\cdot|)$ be a field with an absolute value.
Assume that as a metric space $k$ is complete and separable.
Assume either that $X$ is $T$-transitive or that the image of $|\cdot|\to [0,\infty)$ is closed.
Let $G$ be the $k$-points of an adjoint form simple algebraic group defined over $k$.
Let $c:T\times X \to G$ be a measurable cocycle.
Assume that $c$ is not cohomologous to a cocycle taking values in a proper algebraic subgroup of $G$.
Then either $c$ is cohomologous to a cocycle taking values in a bounded subgroup of $G$,
or there exists a continuous homomorphism $d:T\to G$
such that $c$ is cohomologous to $d$.
\end{theorem}

%%%%%%%%%%%%%

\begin{proof} %[Proof of Theorems \ref{secthm}]
We let $S=T$, $Y=T\times X$ and endow $Y$ with the $S\times T$ action
where $(s,t)$ acts by $(t',x) \mapsto (st't^{-1},sx)$.
Then $Y\ec S\simeq X$ as a $T$ space and by the mixing assumption,
$(Y \ec S)^j$ is $T_i$ ergodic for every $i=0,1,2,\ldots$ and every $j\in \bbN$.
Also $Y\ec T\simeq X$ as an $S$-space and $c:S\times Y\ec T\to G$ is a measurable cocycle
such that the $S$ action on $Y\ec T$ is $c$-ergodic.
The theorem will follow from
Theorem~\ref{main alg}
once we show that the $(S,T_0)$ bi-action on $Y$ is not $c$-ergodic.

We let $B=T/T_0$.
By the Examples \ref{amen subgroup} and \ref{example}, $B$ is an amenable and metrically ergodic $T$-Lebesgue space.
By Theorem~\ref{nontrivial}
the $T$-action on $B\times X$ is not $f$-ergodic,
where $f:T\times B\times X\to G$ is defined by $f(t,b,x)=c(t,x)$.
That is, there exists a proper $k$-algebraic subgroup ${\bf H}\lneq {\bf G}$ and a measurable map
$\phi:B\times X \to {\bf G}/{\bf H}(k)$ satisfying for every $t\in T$, a.e.\ $(b,x)\in  B\times X$, $\psi(tb,tx)=c(t,x)\psi(b,x)$.

Setting ${\bf L}=\{e\}<N_{\bf G}({\bf H})/{\bf H}$, ${\bf V}={\bf G}/{\bf H}$, $d:T_0 \to {\bf L}(k)=\{e\}$ be the trivial map
and $\phi=\psi\circ \theta$, where $\theta:T\times X\to T/T_0\times X$ is the obvious map,
we obtain a non-trivial coset bi-algebraic representation of the $(S,T_0)$ bi-action on $Y$ via the cocycle $c$,
thus indeed, the $(S,T_0)$ bi-action on $Y$ is not $c$-ergodic.
\end{proof}

%%%%%%%%%%%%%%%%%%%%%%%

%%%%%%%

%%%%%%%%%%%%%%%%

\begin{theorem}[Zimmer super-rigidity for arbitrary fields] \label{zimmercor}
Let $l$ be a local field.
Let $T$ to be the $l$-points of a connected almost-simple algebraic group defined over $l$.
Assume that the $l$-rank of $T$ is at least two.
Let $X$ be a $T$-Lebesgue space with finite invariant ergodic measure.

Let $(k,|\cdot|)$ be a field with an absolute value.
Assume that the image of $|\cdot|\to [0,\infty)$ is closed and that as a metric space $k$ is complete and separable.
Let $G$ be the $k$-points of an adjoint form simple algebraic group defined over $k$.
Let $c:T\times X \to G$ be a measurable cocycle.
Assume that $c$ is not cohomologous to a cocycle taking values in a proper algebraic subgroup of $G$.
Then either $c$ is cohomologous to a cocycle taking values in a bounded subgroup of $G$,
or there exists a continuous homomorphism $d:T\to G$
such that $c$ is cohomologous to $d$.
\end{theorem}

%%%%%%%%%%%%%%

\begin{proof}%[Proof of Theorem~\ref{zimmercor}]
By Howe-Moore theorem the action of $T$ on $X$ is mixing.
We let  $A<T$ be a maximal $l$-split torus.
By \cite[Proposition~1.2.2]{margulis-book}
there exists a positive integer $n$ and 1-dimensional subtori $A_0,\ldots,A_n<A$ such that
$T=Z_T(A_0)Z_T(A_n)\cdots Z_T(A_n)$.
We set $T_0=A_0$, $T_1=Z_T(A_0)$, $T_2=A_0$, $T_3=A_1$, $T_4=Z_T(A_1)$, $T_5=A_1$, $T_6=A_2$,
... $T_{3n}=A_n$, $T_{3n+1}=Z_T(A_n)$.
Then
$T$ is generated as a topological group by the closed, non-compact subgroups $T_0,T_1,T_2,\ldots$
and for each $i=1,2,\ldots$ the groups $T_{i-1}$ and $T_i$ commute.
Clearly $T_0$ is amenable and by \cite{BG}  $T/T_0$ admits no non-trivial $T$-invariant semi-metric,
as $T_0$ is not precompact.
The corollary now follows from Theorem~\ref{secthm}.
\end{proof}

%%%%%%%%%%%%%%%%%%%%%%%%%%%

%%%%%%%%%%%%%%%%%%%%%%%%%%%

\begin{theorem} \label{lattice}
Let $T$ be a locally compact second countable group.
Assume $T$ is generated as a topological group by the closed, non-compact subgroups $T_0,T_1,T_2,\ldots$
(for any finite number or countably many $T_i$'s)
such that for each $i=1,2,\ldots$ the groups $T_{i-1}$ and $T_i$ commute.
Assume moreover that $T_0$ is amenable and $T/T_0$ admits no non-trivial $T$-invariant semi-metric.
Let $\Gamma<T$ be a closed subgroup
such that $T/\Gamma$ has a finite $T$-invariant mixing measure.

Let $k$ be a field with an absolute value.
Assume that as a metric space $k$ is complete.
Let $G$ be the $k$-points of an adjoint form simple algebraic group defined over $k$.
Let $\delta:\Gamma \to G$ be a homomorphism.
Assume $\delta(\Gamma)$ is Zariski dense and unbounded.
Then there exists a continuous homomorphism $d:T\to G$
such that $\delta=d|_{\Gamma}$.
\end{theorem}

\begin{proof} %[Proof of Theorem \ref{lattice}]
First observe that there exists a subfield $k'<k$ which is separable and complete, ${\bf G}$ is defined over $k'$ and $\delta(\Gamma)<{\bf G}(k')$.
Indeed, take $k'$ to be the closure of a finite extension of the subfield generated by $\delta(\Gamma)$ over which ${\bf G}$ is defined.
Replacing $k$ by $k'$ we may assume $k$ is separable.
We do so.

We let $\kappa$ be as in Example~\ref{stand} and consider the cocycle $c=\delta\kappa$.
By Proposition~\ref{Gamma-c-ergodic},
$c$ is not cohomologous to a cocycle taking values in a proper algebraic subgroup or a bounded subgroup.
It follows by Theorem~\ref{secthm} that
there exists 
a continuous homomorphism $d:T\to G$
such that $c$ is cohomologous to $d$.
By Proposition~\ref{lattice-vs-cocycle} $\delta$ extends to a homomorphism $T\to G$.
\end{proof}

\begin{proof}[Proof of Theorem~\ref{marguliscor}]
Theorem~\ref{marguliscor} follows from Theorem~\ref{lattice}
exactly in the same fashion that Corollary~\ref{zimmercor}]
follows from Theorem~\ref{secthm}.
\end{proof}

\section{Preliminaries on proper actions} \label{proper prelim}

In this section
we fix a Polish group $G$.
we will discuss the category of proper actions of $G$.
A proper action of $G$ is a Polish space $V$ endowed with an action of $G$
such that the map
\[ G\times V \to V\times V, \quad (g,v) \mapsto (v,gv) \]
is continuous and proper.
The properness assumption is equivalent to the assumption that for every precompact subsets $C_1,C_2\subset V$, the set $\{g~|~gC_1\cap C_2\neq\emptyset\}$
is precompact in $G$.

\begin{prop} \label{closed orbits}
If $V$ is a proper $G$ space than the $G$ orbits in $V$ are closed and the stabilizers are compact.
Furthermore, the orbit space $V/G$ is Hausdorff and
for every $v\in V$ the orbit map $G/G_v\to Gv$ is a homeomorphism.
\end{prop}

\begin{proof}
The fact that the stabilizers are compact follows at once from the definitions (taking $C_1=C_2$, a singleton),
noting that stabilizers are closed.
All remaining parts of the proposition follow from the fact that $V/G$ is Hausdorff
(the last statement follows from Effros theorem, c.f Proposition~\ref{polishing} and the discussion preceding it).
This is a standard fact, which we now briefly recall:
any proper map into a metrizable space is closed, thus the set
$V\times V-\mbox{Im}(G\times V)$ is open in $V\times V$
and its image under the open map to $V/G\times V/G$ is again open
- thus its complement, i.e the diagonal, in the latter is closed.
\end{proof}

The following is a direct corollary of Proposition~\ref{coset contraction-p}.

\begin{prop} \label{coset proper}
Let $T$ be an lcsc group and $X$ an ergodic $T$-Lebesgue space.
Let $c:T\times X\to G$ be a measurable cocycle.
Let $V$ be a proper $G$ action
 and $\phi:X \to V$ a measurable map which is $c$-equivariant.
Then there exists a compact subgroup $K<G$,
a continuous map $i:G/K \to V$
and a $c$-equivariant measurable map $\psi:X\to G/K$
such that $\phi=i\circ \psi$.
\end{prop}

The following proposition is an analogue of Proposition~\ref{mixing}.

\begin{prop} \label{mixing proper}
Let $V$ be a proper $G$ action.
Let $T$ be an lcsc group and $d:T\to G$ a continuous homomorphism with dense image.
Let $X$ be a $T$-Lebesgue space such that
$X^2$ is $T$-ergodic.
Then any equivariant measurable map $\phi:X \to V$ is essentially constant and its essential image is a $G$-fixed point.
\end{prop}

\begin{proof}
We let $K<G$, $i:G/K\to V$ and $\psi:X\to G/K$
be as guaranteed by Proposition~\ref{coset proper}.
We will show that $G=K$, which clearly implies the proposition.
The support of $\psi^2(X^2)$ in $(G/K)^2$ is $d(T)^2$-invariant, thus also $G^2$-invariant.
It follows that the support is $(G/K)^2$.
But by another use of Proposition~\ref{coset proper}, $\psi^2(X^2)$
is supported on a single orbit, which is closed by Proposition~\ref{closed orbits}.
It follows that $(G/K)^2$ consists of only one orbit,
thus indeed $G=K$.
\end{proof}
We end this section with the following technical lemma which will be needed in the proof of Theorem~\ref{initial proper}
below.

\begin{lemma} \label{kappa}
Let $K$ be a second countable compact group.
Let $\kappa$ be an ordinal and assume we are given for every ordinal $\alpha<\kappa$
a $K$-unitary representation $H_\alpha$ and a $K$-unitary injection $H_\alpha\to L^2(K)$.
Assume also that for every $\alpha<\beta<\kappa$ there exists a $K$-unitary injection which is not an isomorphism $H_\alpha\to H_\beta$
(we do not assume any compatibility of these maps).
Then $\kappa$ is a countable ordinal.
\end{lemma}

\begin{proof}
By Peter-Weyl Theorem, $L^2(K)\simeq \hat{\oplus} \pi^{d_\pi}$ where the sum runs over a
collection of representatives of isomorphism classes of irreducible representations of $K$.
This collection is countable, as $K$ is second countable.
For every such $\pi$ we set
\[ m_\pi=\max \{m\leq d_\pi~|~\exists \alpha<\kappa, ~\pi^m<H_\alpha\} \quad \mbox{ and }  \quad
\alpha_\pi=\min\{\alpha<\kappa~|~\pi^{m_\pi}<H_\alpha\}. \]
We let $\kappa'=\sup \alpha_\pi$. This is a countable ordinal.
If $\kappa'+1<\kappa$ the injection $H_{\kappa'}\to H_{\kappa'+1}$ must be an isomorphism.
We conclude that $\kappa\leq \kappa'+1$, hence countable.
\end{proof}

\section{Bi-proper representations} \label{proper reps}

Throughout this section we fix
\begin{itemize}
\item lcsc groups $S$ and $T$,
\item an $(S\times T)$-Lebegue space $Y$,
\item a Polish group $G$,
\item a measurable cocycle
$c:S \times Y\ec T \to G$. We denote by $\tilde{c}$ the pullback of $c$ to $Y$.
\end{itemize}

\begin{defn}
Given all the data above, a bi-proper representation of $Y$
consists of the following data
\begin{itemize}
\item a compact second countable group $L$,
\item a proper $G\times L$ action $V$,
\item a homomorphism $d:T\to L$ with a dense image,
\item a measurable map $\phi:Y \to V$ such that for every $s\in S$, $t\in T$, for almost every $y\in Y$,
\[ \phi(sty)=\tilde{c}(s,y)d(t)\phi(y). \]
\end{itemize}
We abbreviate the notation by saying that $V$ is a bi-proper representation of $Y$,
denoting the extra data by $L_V, d_V$ and $\phi_V$.
A morphism of bi-proper representations from the bi-proper representation $U$ to the bi-proper representation $V$ consists of
\begin{itemize}
\item
a continuous homomorphism $f:L_U\to L_V$ such that $d_V=f\circ d_U$,
\item a continuous map $\psi:U\to V$ which is $G$ and $f$ equivariant,
and such that $\phi_V=\psi\circ \phi_U$.
\end{itemize}
A bi-proper representation of $Y$ is said to be a coset bi-proper representation
if in addition
$V=G/K$ for some compact subgroup $K<G$
and $L<N_G(K)/K$ acts on the right.
\end{defn}

It is clear that the collection of bi-proper representations of $Y$ and their morphisms form a category,
possibly empty.

\begin{theorem} \label{initial proper}
Assume $(Y \ec S)^2$ is $T$-ergodic.
Then the category of bi-proper representations of $Y$, if non-empty, has an initial object.
Moreover, there exists an initial object which is a coset bi-proper representation.
\end{theorem}

We will first prove the following lemma which is a strengthening of
Proposition~\ref{coset proper}.

\begin{lemma} \label{proper mixing coset contraction}
Assume $(Y \ec S)^2$ is $T$-ergodic.
Let $V$ be a bi-proper representation of $Y$.
Then there exists a coset bi-proper representation of $Y$ for some compact group $K<G$
and a morphism of bi-proper representations $i:G/K \to V$.
\end{lemma}

\begin{proof}
By Proposition~\ref{coset proper} we may reduce to the case
$V=(G\times L)/M$ for some compact subgroup $M<G\times L$.
Denote the obvious projection from $G \times L$
to $G$ and $L$ correspondingly by $\pi_1$ and $\pi_2$.
Applying Proposition~\ref{mixing proper}
to the compact group $L$, the proper $L$ space $L/\pi_2(M)$ and the $T$-Lebesgue space
$X=Y\ec S$, we conclude that $\pi_2(M)=L$.
It follows that as $G$-spaces, $ (G\times L)/M \simeq G/\pi_1(M)$.
The lemma follows easily.
\end{proof}

\begin{proof}[Proof of Theorem~\ref{initial proper}]
We will assume throughout that the category of bi-proper representations of $Y$ is non-empty.
By Lemma~\ref{proper mixing coset contraction} we also get that the subcategory of all coset bi-proper representations is non-empty.
We fix a coset bi-representation $\phi_1:Y\to G/K_1$
(we omit here and in what follows $d$ and $L$ from our notation, for simplicity).
We first claim that the subcategory of coset bi-representations has a minimal object:
a coset bi-representation $\phi_0:Y\to V_0$ such that for every coset bi-representation $\phi:Y \to V$, every morphism
of bi-representations, $\psi:V\to V_0$, is an isomorphism.

Assuming not, we will derive a contradiction by transfinite induction.
We will construct inductively
for each countable ordinal $\alpha$ a coset bi-representation $\phi_\alpha:Y\to V_\alpha$
and for each countable $\beta>\alpha$ a morphism of bi-representations which is not an isomorphism
$\psi_{\beta,\alpha}:V_\beta\to V_\alpha$
such that
for any $\gamma>\beta>\alpha$, $\psi_{\gamma,\alpha}=\psi_{\beta,\alpha}\circ\psi_{\gamma,\beta}$.
We start by setting $V_1=G/K_1$.

Assume now
$\delta$ is a countable ordinal such that
for each ordinal $\alpha <\delta$
a coset bi-representation $\phi_\alpha:Y\to V_\alpha$ was already chosen
and also for each $\alpha < \beta < \delta$ a morphism of bi-representations which is not an isomorphism
$\psi_{\beta,\alpha}:V_\beta\to V_\alpha$ was already chosen
such that
for every $\alpha <\beta < \gamma < \delta$, $\psi_{\gamma,\alpha}=\psi_{\beta,\alpha}\circ\psi_{\gamma,\beta}$.

Assume $\delta$ is a successor ordinal, $\delta=\gamma+1$ for some $\gamma$.
By our contradiction assumption, $V_\gamma$
is not a minimal coset bi-representation of $Y$.
We choose
a coset bi-representation $\phi_\delta:Y\to V_\delta$ and a morphism $\psi_{\delta,\gamma}:V_\delta\to V_\gamma$ which is not an isomorphism.
We set for $\alpha < \gamma$, $\psi_{\delta,\alpha}=\psi_{\gamma,\alpha}\circ\psi_{\delta,\gamma}$.
Clearly, $\psi_{\delta,\alpha}$ is not an isomorphism.

Assume $\delta$ is a limit ordinal.
By the countability of $\delta$, the product space $V=\prod_{\alpha<\delta} V_\alpha$ is Polish
and thus gives a proper bi-representation of $Y$.
Applying Lemma
\ref{proper mixing coset contraction} to $V$ we obtain a coset bi-representation $\phi_\delta:Y\to V_\delta$ and a map
$i:V_\delta\to V$. For each $\alpha < \delta$ we set $\psi_{\delta,\alpha}=p_{\alpha}\circ i$.
We clearly have
for every $\alpha < \beta < \delta$, $\psi_{\delta,\alpha}=\psi_{\beta,\alpha}\circ\psi_{\delta,\beta}$.
In particular, for every $\alpha < \delta$, $\psi_{\delta,\alpha}=\psi_{\alpha+1,\alpha}\circ\psi_{\delta,\alpha+1}$
is not an isomorphism, as $\psi_{\alpha+1,\alpha}$ is not an isomorphism.

We set $\kappa=\omega_1$, the first uncountable ordinal.
We had constructed a coset bi-representation for every $\alpha<\kappa$.
We now fix for each such an $\alpha$ a point
$v_\alpha\in \psi_{\alpha,1}^{-1}(eK_1)$ and we set $K_\alpha=\Stab_G(v)$
and $H_\alpha=L^2(K_1/K_\alpha)$.
Clearly we get a $K_1$-unitary injection $H_\alpha\to L^2(K_1)$.
Note that the isomorphism type of $H_\alpha$ does not depend on the choice of $v_\alpha$.
In particular, choosing for $\beta > \alpha$, $v_\alpha=\psi_{\beta,\alpha}(v_\beta)$
we obtain a $K_1$-unitary injection which is not an isomorphism $H_\alpha\to H_\beta$.
By Lemma~\ref{kappa} we obtain that $\kappa$ is countable, which is an absurd.

By this we have proven the existence of a minimal coset bi-representation.
We fix such a minimal coset bi-representation $\phi_0:Y\to G/K_0$.
We now argue to show that this coset bi-representation is an initial object in the category of all bi-proper representations of $Y$.
The argument is similar to the one given in the proof of Theorem~\ref{initial}.

Fix any bi-proper representation of $Y$, $V$.
It is clear that, if exists, a morphism of bi-proper representations from $G/K_0$ to $V$ is unique.
We are left to show existence.
To this end we consider
the product bi-representation $V\times G/K_0$.
Applying Lemma~\ref{proper mixing coset contraction} to this product bi-representation we obtain the commutative diagram

%\medskip

\begin{equation} \label{diag-AG}
\xymatrix{ Y \ar@{.>}[r] \ar[d]^{\phi_V} \ar[rd]^{\phi} \ar@/^3pc/[rrd]^{\phi_0} & G/K \ar@{.>}[d]_{i} &  \\
		   V & V\times G/K_0 \ar[r]^{~~p_2} \ar[l]_{p_1~~~~} & G/K_0  }
\end{equation}

By the fact that $G/K_0$ is a minimal object in the category of coset bi-representations,
the morphism $p_2\circ i$ must be an isomorphism.
We thus obtain the morphism
\[ p_1\circ i \circ (p_2\circ i)^{-1}:G/K_0 \to V. \]
It is easy to check that this is a morphism of bi-proper representations.
\end{proof}

The following proposition is an analog of Proposition~\ref{T'}.
Its proof is essentially the same and we will not repeat it.

\begin{prop} \label{T' proper}
Let $T'$ be a group acting on $Y$ commuting with $S$ and $T$.
Assume the cocycle $\tilde{c}$ is $T'$-invariant in the sense that
for all $s\in S$, $t'\in T'$ and a.e.\ $y\in Y$, $\tilde{c}(s,ty)=\tilde{c}(s,y)$.
Assume $(Y \ec S)^2$ is both $T$ and $T'$ ergodic.
Let $G/K$, $L<N_G(K)/K$,
$d:T \to L$ and $\phi:Y \to G/K$ be an initial object
in the category of bi-proper representations of the $(S,T)$ bu-space $Y$, as guaranteed by Theorem~\ref{initial proper}.
Then there exists a compact subgroup
$L'<N_G(K)/K$
which commutes with $L$
and a homomorphism $d':T'\to L'$
such that $G/K$, $L'$, $d'$ and $\phi$ form an initial object in the category of bi-proper representations of the $(S,T')$ bi-space $Y$.
\end{prop}

\section{Super-Rigidity over local fields} \label{local field}

\begin{prop}
Let $k$ be a local filed.
Let ${\bf G}$ be a $k$-algebraic group and ${\bf V}$ a $k$-${\bf G}$-affine variety.
Denote $G={\bf G}(k)$ and $V={\bf V}(k)$.
Let $C\subset V$ be a compact subset and denote by $\bar{B}$ its Zariski closure.
Then the group $\Stab_G(B)/\Fix_G(\bar{B})$ is compact.
\end{prop}

\begin{proof}
Without loss of generality we may replace ${\bf G}$ by $\Stab_{\bf G}(\bar{B})$ and then assume ${\bf V}=\bar{B}$.
We then may further assume ${\bf G}=\Stab_{\bf G}(\bar{B})/\Fix_{\bf G}(\bar{B})$. We do so.
By \cite[Proposition~1.12]{borel} there exists an embedding $V\to k^n$, which we may assume spanning,
equivariant with respect to some representation $G\to\GL_n(k)$, which we thus may assume injective.
Denote by $D$ the image of $C$ in $V$ and by $K$ the image of $\Stab_G(B)$ in $\GL_n(k)$.
Then $K$ preserves the balanced convex hull of $D$ given by
\[ E=\left\{\sum_{i=1}^n \alpha_iv_i~|~v_i\in D,~\alpha_i\in k,~\sum_{i=1}^n |\alpha_i|\leq 1\right\}, \]
and the associated Minkowski norm on $k^n$ defined by
\[ \|x\|=\sup\{r>0~|~\forall \alpha\in k,~|\alpha|<r~\Rightarrow~\alpha x\in E\}. \]
Thus $K$ is compact.
\end{proof}

By applying the proposition to the conjugation action of ${\bf G}$ on itself, we obtain the following.

\begin{cor} \label{NL}
Let $k$ be a local filed.
Let ${\bf G}$ be a $k$-algebraic group and denote $G={\bf G}(k)$.
Let $K<G$ be a compact subgroup.
Then $N_G(K)/Z_G(K)$ is compact.
In particular, if ${\bf G}$ is an adjoint form semisimple group and $K$ is Zariski dense then $N_G(K)$ is compact.
\end{cor}

\begin{theorem} \label{main}
Let $S$ and $T$ be locally compact second countable groups.
Assume $T$ is generated as a topological group by the closed, non-compact subgroups $T_0,T_1,T_2,\ldots$
(for finitely or countably many $T_i$'s)
such that for each $i=1,2,\ldots$ the groups $T_{i-1}$ and $T_i$ commute.
Let $Y$ be an $S\times T$ Lebesgue space.
Denote by $Y\ec S$ the space of ergodic components of $Y$ with respect to the $S$-action.
This is a $T$-Lebesgue space.
Assume $Y$ is $T_0\times S$ amenable
and $(Y \ec S)^j$ is $T_i$ ergodic for every $i=0,1,2,\ldots$ and every $j\in \bbN$.

Let $k$ be a local field.
Let $G$ be the $k$-points of an adjoint form simple algebraic group defined over $k$.
Let $c:S\times Y \ec T \to G$ be a measurable cocycle.
Assume that $c$ is not cohomologous to a cocycle taking values in a proper algebraic subgroup or a
bounded subgroup of $G$.
Then there exists a continuous homomorphism $d:T\to G$ and a measurable map
$\phi:Y\to G$ with the following property:
for every $s\in S$, $t\in T$, for almost every $y\in Y$,
\[ \phi(sty)=\tilde{c}(s,y)\phi(y) d(t)^{-1},\]
where $\tilde{c}:S\times Y\to G$ is the pullback of $c$.
\end{theorem}

\begin{proof} %[Proof of Theorem~\ref{main}]
The theorem will follow from
Theorem~\ref{main alg}
once we show that
the $(S,T_0)$ bi-action on $Y$ is not $c$-ergodic.
Our standing assumption for the rest of the proof is that the $(S,T_0)$ bi-action on $Y$ is $c$-ergodic.
We will derive a contradiction by showing that the cocycle $c$
is cohomologous to a cocycle taking values in a bounded subgroup of $G$,
which is,
by Proposition~\ref{proper-cohom},
equivalent to the existence of
a $c$-equivariant map $\psi:Y \ec T \to G/N$ where $N<G$ is bounded.

By Theorem~\ref{BDL} for $R=S\times T_0$ and $f=\tilde{c}$
(that is for $s\in S, t\in T_0$ and $y\in Y$, $f(s,t,y)=\tilde{c}(s,y)$)
we obtain that the category of bi-proper representations of the $(S,T_0)$ bi-action on $Y$ is not empty.
By Theorem~\ref{initial proper} this category has an initial object which is a coset bi-representation.
Let $G/K$, $L_0<\Aut_G(G/K)$,
$d_0:T_0 \to L_0$ and $\phi:Y \to G/K$ form such an initial object.
By our standing assumption
$K$ is Zariski dense in ${\bf G}$,
thus, by Corollary~\ref{NL}, $N=N_G(K)$ is compact.
By Proposition~\ref{T' proper}
there also exist
a compact subgroup
$L_1<N/K$
which commutes with $L_0$
and a homomorphism $d_1:T_1\to L_1$
such that $G/K$, $L_1$, $d_1$ and $\phi$ form an initial object in the category of bi-proper representations of the $(S,T_1)$ bi-action on $Y$.

Repeating this argument for each pair of groups $T_{i-1},T_i$ we get
compact subgroups $L_i<N/K$
and homomorphisms
$d_i:T_i\to L_i$
satisfying for every $t\in T_i$ and almost every $y\in Y$,
$\phi(ty)=\phi(y)d_i(t)^{-1}$.

We conclude that the $G$-morphism $\pi:G/K \to G/N$
is $L_i$ invariant for every $i$.
Since $T$ is topologically generated by the groups $T_i$ and $L_i=\overline{d_i(T_i)}$,
it follows that $\pi\circ \phi:Y \to G/N$ factors through
$Y \ec T$ and we get indeed a $c$-equivariant map $\psi:Y \ec T \to G/N$.
\end{proof}

%%%%%%%%%%%%%%!!!!!!!!!!!!!!!!!!

\begin{theorem} \label{firstcor}
Let $S$ and $T$ be locally compact second countable groups.
Assume $T$ is generated as a topological group by the closed, non-compact subgroups $T_0,T_1,T_2,\ldots$
(for any finite number or countably many $T_i$'s).
Assume $T_0$ is amenable and for each $i=1,2,\ldots$ the groups $T_{i-1}$ and $T_i$ commute.
Let $Y$ be a measured coupling of $S$ and $T$ in the following sense:
$S\times T$ acts on $Y$, as a $T$-space $Y \simeq X\times T$ for some space $X$ where the $T$-action is on the
second coordinate and
as an $S$-space $Y \simeq X'\times S$ for some space $X'$ where the $S$-action is on the
second coordinate.
Consider $X$ as an $S$-space and
let $e:S\times X\to T$ be the associated cocycle.
Assume
that as a $T$-space $X'$ has a finite invariant mixing measure.

Let $k$ be a local field.
Let $G$ be the $k$-points of an adjoint form simple algebraic group defined over $k$.
Let $c:S\times X \to G$ be a measurable cocycle.
Assume that $c$ is not cohomologous to a cocycle taking values in a proper algebraic subgroup or a
bounded subgroup of $G$.
Then there exists a continuous homomorphism $d:T\to G$
such that $c$ is cohomologous to $d\circ e$.
\end{theorem}

\begin{proof} %[Proof of Theorem~\ref{firstcor}]
By setting $T_1=\{e\}$, $T_2=S$, $B_1=X'$ and $B_2=S$ in Proposition~\ref{product-stable}
we get that the $S$ action on $Y$ is amenable.
Given any measurable $S\times T_0$-bundle of convex compact space over $Y$, $\pi:C\to Y$, we consider the space $L^0(\pi)$
consisting of classes of measurable sections defined up to null sets.
By the amenability of the $S$ action on $Y$ we conclude that the space of $S$-invariants, $L^0(\pi)^S$,
is not empty.
This space has a natural convex compact structure and it is preserved by the $T_0$-action.
By the amenability of $T_0$ there is a fixed point in that space, namely an $S\times T_0$ invariant section for $\pi$.
We conclude that the action of $S\times T_0$ on $Y$ is amenable.
The corollary now follows from Theorem~\ref{main}.
\end{proof}

\begin{proof}[Proof of Theorem~\ref{prezimmer}]
This is the special case of Corollary~\ref{firstcor} where we choose $S=T$ and $Y=X\times T$
endowed with the $T\times T$ action where $(t_1,t_2)$ acts by $(x,t)\mapsto (t_1x,t_1tt_2^{-1})$.
\end{proof}

%%%%%%%%%%%%%%%%%%%%%%%%%%%%%%%%%%%%%

%%%%%%%%%%%%%%%%%%%%%%%%%%%%%%%%%%%%%%%%%%%

%%%%%%%%%%%%%%%%%%%%%%%%%%%%%%%%%%%

\end{document}